\documentclass[12pt,a4paper,twoside,reqno]{amsart}

\pdfoutput=1 
	
		\usepackage{geometry} 
		\usepackage{enumitem} 

		\usepackage[table,xcdraw]{xcolor} 
		\definecolor{navy}{HTML}{000080}
		\usepackage{hyperref}
		\hypersetup{
		  colorlinks   = true, 
		  urlcolor     = navy, 
		  linkcolor    = navy, 
		  citecolor   = navy 
		}

		\usepackage[super]{nth} 

		\usepackage[verbose]{wrapfig} 
		\usepackage{caption} 
		\usepackage{subcaption} 
		\numberwithin{figure}{section}

		\usepackage{changepage} 

		\usepackage{tikz-cd} 

		\usepackage{verbatim} 

		\usepackage[UKenglish]{babel} 

		\let\oldtocsection=\tocsection
		\renewcommand{\tocsection}[2]{\bf\oldtocsection{#1}{#2}} 
		\let\oldtocsubsection=\tocsubsection
		\renewcommand{\tocsubsection}[2]{\hspace{1em}\oldtocsubsection{#1}{#2}} 

		\usepackage{pinlabel} 

		\usepackage{bm} 

		\usepackage{booktabs} 
		\usepackage{longtable} 
		\usepackage{array} 
		\usepackage[export]{adjustbox}

		\usepackage{floatrow} 

		\usepackage{calrsfs} 

			\newcommand{\ignore}[1]{}
			\newcommand{\R}{\mathbb{R}}
			
			\newcommand{\Z}{\mathbb{Z}}
			\newcommand{\N}{\mathbb{N}}

			\usepackage{colonequals}
				\newcommand{\defeq}{\colonequals}
			\newcommand{\surj}{\twoheadrightarrow}
			\newcommand{\sing}[1]{#1_\mathrm{s}}
			\newcommand{\take}{\setminus}
			
			\newcommand{\setb}[2]{\left\{#1~\middle|~#2\right\}}
			
			\newcommand{\mathnth}[1]{#1^\mathrm{th}}
			
			\newcommand{\ucov}{p}
			\newcommand{\reg}[1]{#1\take\sing{#1}}
			
			\newcommand{\link}{\Sigma}

			\renewcommand{\epsilon}{\varepsilon}
			\renewcommand{\phi}{\varphi}
			\renewcommand{\theta}{\vartheta}
			\renewcommand{\S}{\mathbb{S}}
			\renewcommand{\rm}{\mathrm}
			\usepackage{amssymb}
				\renewcommand{\emptyset}{\varnothing}

			\renewcommand{\angle}{\measuredangle}

			\DeclareMathOperator{\SO}{SO}
			
			\DeclareMathOperator{\GL}{GL}
			\DeclareMathOperator{\Orth}{O}
			
			\DeclareMathOperator{\mult}{mult}

			\DeclareMathOperator{\diam}{diam}
			\DeclareMathOperator{\Mon}{Mon}
			\DeclareMathOperator{\Vol}{Vol}
			
			\DeclareMathOperator{\dev}{Dev}

			\DeclareMathOperator{\pack}{pack}

			\usepackage{mathtools}

			\DeclareMathAlphabet{\oldcal}{OMS}{zplm}{m}{n} 

		\usepackage{thm-restate}

		\usepackage[capitalise,noabbrev,nameinlink]{cleveref}
		  \usepackage{crossreftools}
		  \usepackage{xpatch}
		  \ExplSyntaxOn
		  \cs_new_eq:NN \ExpandableMakeLowercase \text_lowercase:n
		  \ExplSyntaxOff
		  \makeatletter
		  \xpatchcmd{\@othm}{\MakeLowercase}{\ExpandableMakeLowercase}{}{}
		  \xpatchcmd{\@ynthm}{\MakeLowercase}{\ExpandableMakeLowercase}{}{}
		  \makeatother

		\newtheorem{thm}{Theorem}[section]
		  \crefname{thm}{Theorem}{Theorems} 
		\newtheorem{conj}[thm]{Conjecture}
			\crefname{conj}{Conjecture}{Conjectures}
		\newtheorem{prop}[thm]{Proposition}
		  \crefname{prop}{Proposition}{Propositions}
		\newtheorem{lem}[thm]{Lemma}
		  \crefname{lem}{Lemma}{Lemmas}
		\newtheorem{cor}[thm]{Corollary}
		  \crefname{cor}{Corollary}{Corollaries}

		\theoremstyle{definition}
		\newtheorem{defn}[thm]{Definition}
		  \crefname{defn}{Definition}{Definitions}
		\newtheorem{rem}[thm]{Remark}
		  \crefname{rem}{Remark}{Remarks}
		
		  \crefname{exmp}{Example}{Examples}
		\newtheorem*{exmp*}{Example}
		  \crefname{exmp*}{Example}{Examples}
		\newtheorem{proc}[thm]{Procedure}
		  \crefname{proc}{Procedure}{Procedures}
		\newtheorem*{nota*}{Notation}

		\theoremstyle{remark}

		\numberwithin{equation}{section}
		\crefname{equation}{Formula}{Formulae}

		\usepackage[headings]{fullpage}
		\usepackage{microtype}

\begin{document}

\title[Singular Vertices]{Singular vertices of nonnegatively curved integral polyhedral 3-manifolds}
\author{Thomas Sharpe}
\date{10th February 2023}
\address{Thomas Sharpe, Department of Mathematics, King's College London, Strand, London, WC2R 2LS}
\email{tom.sharpe@kcl.ac.uk}

\begin{abstract}
In this paper, we study polyhedral 3-manifolds with nonnegative curvature and integral monodromy, two conditions motivated by Thurston's work in \cite{thurston1998shapes}. We classify the 32 isometry types of codimension 3 singularities in such manifolds. We also show that the number of these singularities is bounded.
\end{abstract}
\vspace*{-25pt}
\maketitle

\section{Introduction} 
	\label{sec:intro}

	\noindent A \emph{polyhedral \(n\)-manifold} is an \(n\)-manifold with a geodesic metric induced by a Euclidean triangulation. We call a polyhedral manifold \emph{nonnegatively curved} if the sum of the dihedral angles (the \emph{conical angle}) around any codimension 2 simplex is at most \(2\pi\). In \cite{thurston1998shapes}, Thurston made extensive study of nonnegatively curved polyhedral metrics on \(S^2\). In three dimensions, they play an important role in the proof of Thurston's orbifold theorem, as demonstrated by Cooper, Hodgson, and Kerckhoff in \cite{cooper2000three}. More recently, they have been studied by Cooper and Porti in \cite{MR2484703}, and Lebedeva, Matveev, Petrunin, and Shevchishin in \cite{MR3361404}, among others.

	Thurston's work included the study of metrics induced by \emph{equilateral} triangulations. This motivates the notion of \emph{integrality}---essentially, the requirement that parallel transport outside of the singularities preserves a lattice. In two dimensions, singularities are just isolated conical points, while in three dimensions, they form a union of lines, circles, and graphs of minimum degree 3, whose vertices are called \emph{singular vertices}. From the Gauss--Bonnet formula for polyhedral surfaces (for which, see \cite[Thm.~3.15]{cooper2000three}), it follows that the number of conical points of any nonnegatively curved integral polyhedral surface is at most 12. In this paper, we prove the following 3-dimensional generalisation of this bound.

	\begin{restatable}[Singular vertex bound]{thm}{singVertBnd}\label{res:singVertBnd}
		There is a constant \(B_\rm{ver}\in\N\) such that \emph{any} (complete, connected) nonnegatively curved integral polyhedral 3-manifold has fewer than \(B_\rm{ver}\) singular vertices.
	\end{restatable}

	The terms used here are defined more precisely in \cref{sec:conv}. This result follows from \cref{res:epsVertBndAlexandrov}---a more general result about \emph{Alexandrov spaces of nonnegative curvature}, which most of \cref{sec:proof} is dedicated to proving. Specifically, call a point \(x\) in an Alexandrov space \emph{\(\epsilon\)-narrow} if the angle between any two paths meeting at \(x\) is at most \(\epsilon\). \Cref{res:epsVertBndAlexandrov} states that there is a bound on the number of \(\varepsilon\)-narrow points in any nonnegatively curved Alexandrov space depending only on \(\varepsilon\) and the dimension of the space. After this paper was first written, the author learned that \cref{res:epsVertBndAlexandrov} actually follows from a recent and very general result of Li and Naber, \cite[Cor.~1.4]{MR4171913}. The proof of \cref{res:epsVertBndAlexandrov} presented in \cref{sec:proof} is independent of their work and, due to the more specialised setting of our problem, is shorter than the proof of their result.

	The reasons that \cref{res:singVertBnd} follows from \cref{res:epsVertBndAlexandrov} are that nonnegatively curved polyhedral 3-manifolds are 3-dimensional Alexandrov spaces of nonnegative curvature (by \cite[Ex.~2.9~(6)]{MR1185284}) and that there is a fixed \(\epsilon_0\in(0,\pi)\) such that every singular vertex in any nonnegatively curved integral polyhedral 3-manifold is \(\epsilon_0\)-narrow. The latter statement follows from the fact that there are only finitely many local isometry types of singular vertices in nonnegatively curved integral polyhedral 3-manifolds (see \cref{res:linkDiameterBoundedAwayFromPi}). In this paper, we show not only this, but fully classify these local isometry types.

	\begin{thm}[Classification of singular vertices]\label{res:classificationOfSingularVertsShort}
		There are 32 possibilities for the local isometry type of a singular vertex in a nonnegatively curved integral polyhedral 3-manifold.
	\end{thm}

	Thurston's work in two dimensions---particularly in \cite[§~6]{thurston1998shapes}---is a major motivation for studying integral metrics in three dimensions. Another motivation comes from the notion of an \emph{integral affine manifold with singularities}. These are manifolds which, away from a codimension 2 subset, have an atlas with transition functions in affine transformations whose linear parts have integer entries (for more details, see \cite[Defs.~3.1~\&~3.6]{MR2487600}). Integral polyhedral manifolds are examples of these. The regular locus of such a manifold is naturally the base of a torus bundle, and compactifications of these bundles are often Calabi--Yau manifolds (see e.g., \cite{MR1882331,MR1975331}). It is a well-known open question, whether or not there are finitely many families of Calabi--Yau threefolds (see \cite{wilson2021boundedness}). By analogy, this question motivates similar boundedness questions for the singular loci of nonnegatively curved integral polyhedral 3-manifolds, and \cref{res:singVertBnd} answers one such question.

	Noncompact polyhedral 3-manifolds with conical angles \(\leq\pi\) have been classified by Boileau, Leeb, and Porti in \cite[Thm.~4.1]{boileau2005geometrization}. In the noncompact case, when the conical angles are \(\leq c<2\pi\) and the singular locus is a submanifold, Cooper and Porti explain in \cite{MR2484703} that they can give an explicit upper bound on the number of singular components, depending only on \(c\). Brief reference is made to monodromy constraints by Porti and Weiss in \cite[§~2]{porti2007deforming}. But, to the author's knowledge, no general singularity bound (such as \cref{res:singVertBnd}) or classification (such as \cref{res:classificationOfSingularVertsShort}) has been produced when integral monodromy is imposed. These results and their proofs bring together Alexandrov geometry, ramified covering theory, and the classification of crystallographic groups.

	\subsection*{Layout} The rest of the paper is organised as follows.
	\begin{itemize}
		\item In \cref{sec:conv}, we make precise some key definitions used throughout the paper, including `polyhedral manifold', `nonnegatively curved', and `integral'.
		\item In \cref{sec:class}, we prove an expanded and more precise form of \cref{res:classificationOfSingularVertsShort}. We do this by relating the local geometry of any point to a ramified cover of one of two spherical orbifolds. We also deduce \cref{res:linkDiameterBoundedAwayFromPi}, a result needed for the next section, which states that there is \(\varepsilon_0\in(0,\pi)\) such that any singular vertex is \(\varepsilon_0\)-narrow.
		\item In \cref{sec:proof}, we prove \cref{res:epsVertBndAlexandrov} via a series of lemmas and then deduce \cref{res:singVertBnd}. The proof is constructive, and key to it are two results from Alexandrov geometry, namely the Bishop--Gromov inequality and Toponogov's theorem.
		\item In \cref{sec:further}, we consider some natural extensions of \cref{res:singVertBnd}.
		\item Finally, the \hyperref[sec:appendix]{appendix} contains the full classification mentioned in \cref{res:classificationOfSingularVertsShort} and some worked examples that illustrate the process by which this classification was produced. It is not technically needed for the exposition, but readers who want more detail than is provided in \cref{sec:class} are directed there.
	\end{itemize}

	\subsection*{Acknowledgements} The author would like to thank his doctoral supervisor, Dmitri Panov, for his support and mathematical insight, without which this paper would have been impossible. The author would also like to thank Anton Petrunin for his ideas that went into \cref{sec:proof}. Many thanks to Martin de Borbon and Steven Sivek for their comments on preliminary versions of the paper. This work was supported by the EPSRC grant EP/L015234/1 through the London School of Geometry and Number Theory and King's College London.

	\subsection{Preliminary Definitions}
	\label{sec:conv}

	Here we make precise a few key definitions used throughout the paper. Specifically, we define \emph{polyhedral \(n\)-manifolds} and concepts relating to \emph{nonnegative curvature} and \emph{integrality}. The notion of a polyhedral manifold is not consistent throughout the literature. For example, in \cite[Def.~3.4.1]{MR3930625}, Alexander, Kapovitch, and Petrunin require a finite triangulation; whereas in \cite[Def.~3.2.4]{burago2001course}, Burago, Burago, and Ivanov do not even require the triangulation to be \emph{locally} finite. Some authors use the term \emph{cone-manifold}, such as Cooper, Hodgson, and Kerckhoff in \cite{cooper2000three}. We adopt the following definition, in line with de Borbon and Panov in \cite[§~6.1.2]{BP2021polyhedral}.

	\begin{defn}\label{defn:polyhedralManifold}
		Let \(M\) be a topological \(n\)-manifold without boundary endowed with a complete, geodesic metric \(d\) (i.e., a complete metric for which the distance between two points is equal to the length of the shortest path joining them). Suppose that \(M\) admits a locally finite triangulation in which each simplex is isometric to a Euclidean simplex. We then call \(M\) a \emph{polyhedral \(n\)-manifold} and \(d\) a \emph{polyhedral metric}. This definition implies that a polyhedral manifold is connected, but not necessarily compact.
	\end{defn}

	\noindent We now define some important concepts relating nonnegative curvature.

	\begin{defn}\label{defn:conicalAngle}
		Let \(M\) be a polyhedral \(n\)-manifold. Each codimension 2 simplex \(\sigma\) of \(M\) has a dihedral angle associated to each \(n\)-simplex to which it belongs. The sum of these dihedral angles is called the \emph{conical angle} of \(\sigma\). The union of all codimension 2 simplices whose conical angles differ from \(2\pi\) is called the \emph{singular locus} of \(M\), denoted by \(\sing M\). Elements of \(\sing M\) are referred to as \emph{singularities} or \emph{singular points}, while the set \(\reg M\) of \emph{regular points} is called the \emph{regular locus} of \(M\). Each singularity has a well-defined \emph{codimension} \(k\in \{2,\ldots,n\}\) (for which, see \cite[§~6.2.2]{BP2021polyhedral}). We say that \(M\) is \emph{nonnegatively curved} if all its conical angles are at most \(2\pi\).
	\end{defn}

	It is not hard to see that, in a polyhedral 3-manifold, singular vertices (as defined on the first page) are precisely codimension 3 singularities. We complete this section by defining monodromy and integrality.

	\begin{defn}\label{defn:monodromy}
		Let \(M\) be a polyhedral \(n\)-manifold. The regular locus of \(M\) has a flat Riemannian metric. Parallel transport inside the regular locus gives rise to a \emph{monodromy map} \(\Mon\colon\pi_1(\reg M)\to\Orth(n)\). The image of this map is called the \emph{monodromy group} of \(M\), denoted by \(\Mon M\), and we say that \(M\) is \emph{integral} if the action of this group on \(\R^n\) preserves a lattice (of full rank). This is equivalent to \(\Mon M\) being conjugate (in \(\GL(n,\R)\)) to a subgroup of \(\GL(n,\Z)\).
	\end{defn}

	In 2 dimensions, integrality simply reduces to the requirement that each conical angle be a multiple of either \(\pi/3\) or \(\pi/2\) (at least in the orientable case). In 3 dimensions, the condition is more subtle and is explored in \cref{sec:class}.


\section{Classification of Singular Vertices} 
	\label{sec:class}

	\noindent Polyhedral manifolds are, by construction, \emph{locally conical}---that is, every point \(x\) in a polyhedral manifold \(M\) has a neighbourhood which is pointed isometric to a small ball about the tip of a polyhedral cone. A \emph{polyhedral cone} is a polyhedral manifold, homeomorphic to \(\R^n\), that is also a metric cone (see \cite[§~3.4]{MR3930625} for more details). The geometry of this cone is completely determined by the unit sphere about its tip, which we call the \emph{link} of \(x\), denoted by \(\link_xM\). Therefore, the local isometry type of \(x\) may be taken to mean the isometry type of \(\link_xM\). The link coincides with the \emph{space of directions} of Alexandrov geometry (for which, see \cite[§~3.6.6]{burago2001course}). When the manifold is 3-dimensional, the link of a point is a \emph{sphere with conical points}---see \cite[Def.~1.1]{mondello2016spherical} for a definition.

	\begin{rem}
		Our convention is to use \(\S^2\) to denote the unit sphere with its intrinsic metric, while using \(S^2\) simply to denote the underlying topological space.
	\end{rem}

	A point in a polyhedral 3-manifold is regular if and only if its link is isometric to \(\S^2\). The local isometry type of a codimension 2 singularity is determined by one number: the conical angle of the edge on which it lies. In the nonnegatively curved \emph{integral} case, the possible angles are given in \cref{res:integralConicalAngles}. The main aim of this section is to classify the local isometry types of codimension 3 singularities, or singular vertices, of nonnegatively curved integral polyhedral 3-manifolds. Our approach requires the following definition.

	\begin{defn}[c.f. {\cite[Def.~1.2.18]{lando2013graphs}}]\label{defn:ramifiedCover}
		Let \(S_1\) and \(S_2\) be spheres with conical points, with a surjection \(\varphi\colon S_1\to S_2\). Suppose that \(\varphi\) is a locally isometric covering map, except at the preimages of the conical points of \(S_2\), where it is locally the quotient by a finite group of isometries. We then say that \(S_1\) is a \emph{ramified cover} of \(S_2\) and refer to \(\varphi\) as the \emph{ramified covering map}. We will usually use a double-headed arrow, as in \(\varphi\colon S_1\surj S_2\), to represent a ramified cover. The \emph{degree} of \(\varphi\) is its degree as a covering map outside of the preimages of the conical points of \(S_2\), denoted by \(\deg\varphi\). In a sufficiently small neighbourhood of any point \(x_1\in S_1\), \(\varphi\) is the quotient by a finite group of isometries, the order of which is called the \emph{multiplicity} of \(x_1\), denoted by \(\mult{x_1}\). This means that the multiplicity of any preimage of a regular point of \(S_2\) is 1. For any \(x_2\in S_2\), the following holds:
		\begin{equation}\label{eq:degMult}
			\sum_{\mathclap{x_1\in\varphi^{-1}(x_2)}}\mult{x_1} = \deg\varphi.
		\end{equation}
	\end{defn}

	To classify the local isometry types in question, we show that each link can be expressed as a ramified cover of one of two spherical orbifolds, with nonnegative curvature implying restraints on the multiplicities. These ramified covers are determined by combinatorial data, which can be enumerated. The techniques for enumeration are explained from \cref{res:conePointsBoundForMonodromyD6} to \cref{proc:findingRamifiedCovers}, and the results are summarised in \cref{res:classificationOfSingularVerts} and given in full in the \hyperref[sec:appendix]{appendix}. The two spherical orbifolds in question come from the possible monodromy groups of integral polyhedral 3-manifolds, so we begin by classifying these.

	\begin{lem}\label{res:possibleMonGroupsIn3D}
		An (orientable) polyhedral 3-manifold is integral if and only if its monodromy group is isomorphic a subgroup of \(S_4\) or \(D_6\); i.e., to \(1\), \(C_2\), \(C_3\), \(C_4\), \(C_6\), \(D_2\), \(D_3\), \(D_4\), \(D_{6}\), \(A_4\), or \(S_4\).
	\end{lem}
	\begin{proof}
		The list given above agrees with the list of 3-dimensional lattice-preserving rotation groups in, e.g., \cite[§~15.6]{coxeter1989introduction}. It follows from the classification of finite subgroups of \(\SO(3)\) (for which, see \cite[Ch.~19]{MR965514}) that such subgroups are isomorphic if and only if they are conjugate. Therefore, it suffices to check the isomorphism type of the monodromy group.
	\end{proof}

	\begin{rem}\label{res:monGroupsActions}
		In the above classification, \(S_4\) acts as the rotational symmetries of a cube, and \(D_6\) as the rotational symmetries of a regular hexagonal prism.
	\end{rem}

	This allows us to list all the possible conical angles of a nonnegatively curved integral polyhedral 3-manifold (thereby classifying the local isometry types of all codimension 2 singularities). 

	\begin{cor}\label{res:integralConicalAngles}
		The possible conical angles (differing from \(2\pi\)) in a nonnegatively curved integral polyhedral 3-manifold are
		\begin{enumerate}
			\item \(\pi/2\), \(2\pi/3\), \(\pi\), \(4\pi/3\), and \(3\pi/2\), when the monodromy is a subgroup of \(S_4\); or
			\item \(\pi/3\), \(2\pi/3\), \(\pi\), \(4\pi/3\), and \(5\pi/3\), when the monodromy is a subgroup of \(D_6\).
		\end{enumerate}
	\end{cor}
	\begin{proof}
		The result follows by looking at possible orders of nontrivial elements of \(S_4\) and \(D_6\)---i.e., 2, 3, and 4; and 2, 3, and 6 respectively---and then looking for rotations by less than \(2\pi\) having such orders. All of them can be exhibited in integral polyhedral metrics on \(S^3\) by taking doubles of prisms.
	\end{proof}

	\begin{wrapfigure}[10]{r}{0pt}
		\begin{tikzcd}
			\widetilde{\link_xM} \arrow[rd, "\dev"] \arrow[dd, "p"'] &                                     \\
			                                                                    & \S^2 \arrow[d, two heads]           \\
			\link_xM \arrow[r, "\exists!", dashed, two heads]		                  	    & \S^2/\Mon_xM \arrow[d, two heads] \\
			                                                                    & \S^2/G                              
		\end{tikzcd}
		\caption{}
		\label{fig:ramifiedCover}
	\end{wrapfigure}
	Now that we have classified monodromy groups and conical angles, we can give the key result that underpins this section.

	\begin{prop}\label{res:ramifiedCoverForLinks}
		Let \(x\) be a singular vertex in a nonnegatively curved integral polyhedral 3-manifold \(M\). Then \(\link_xM\) is a ramified cover of \(\S^2/S_4\) or \(\S^2/D_6\).
	\end{prop}
	\begin{proof}
		Let \(\Mon_xM\) denote the \emph{local} monodromy group of \(M\) at \(x\); that is, the monodromy group of a small neighbourhood of \(x\). We note two key facts about this group. Firstly, it is a subgroup of \(\Mon M\), since a small neighbourhood of \(x\) intersected with \(\reg M\) inherits its Riemannian metric from \(\reg M\). Secondly, it can be viewed as the \emph{holonomy} group of \(\link_xM\) (in the \((\SO(3),\S^2)\)-cone-manifold sense, for which, see \cite[§~4.1]{cooper2000three}). Since \(\Mon_xM\) is a finite subgroup of \(\SO(3)\), the quotient \(\S^2/\Mon_xM\) is an orbifold. We construct a ramified covering map \(\link_xM\surj\S^2/\Mon_xM\) as follows. Let \(\ucov\colon\widetilde{\link_xM}\to \link_xM\) be the completion of the universal cover and \(\dev\colon\widetilde{\link_xM}\to\S^2\) the corresponding developing map. Given \(u\in \link_xM\), consider the set \(\dev(\ucov^{-1}(u))\). If \(v\) and \(v'\in\dev(\ucov^{-1}(u))\), then the definition of holonomy implies that there is some \(g\in\Mon_xM\) such that \(v'=g\cdot v\). The following assignment therefore gives a well-defined ramified covering map:
		\begin{align*}
			\link_xM &\surj\S^2/\Mon_xM,\\
			u &\mapsto	 [v], \text{ where } v\in\dev(\ucov^{-1}(u)).
		\end{align*}

		The above map may be seen as the unique map in \cref{fig:ramifiedCover} making the diagram commute. By \cref{res:possibleMonGroupsIn3D}, \(\Mon_xM\) is a subgroup of \(G=S_4\) or \(D_6\), so we can compose this map with the quotient map \(\S^2/\Mon_xM\surj\S^2/G\) to get a ramified covering map \({\link_xM\surj\S^2/G}\). It is useful to note that \(\S^2/S_4\cong\S^2(\pi/2,2\pi/3,\pi)\) and \(\S^2/D_6\cong\S^2(\pi/3,\pi,\pi)\) (see \cite[Ex.~3.2]{cooper2000three} for an explanation of this notation) and that the areas of these orbifolds are \(\pi/6\) and \(\pi/3\) respectively.
	\end{proof}

	\begin{rem}\label{res:linksAreRamifiedCovers}
	This result implies that classifying links of singular vertices is equivalent to classifying spheres with at least three conical points and with conical angles at most \(2\pi\) that are ramified covers of \(\S^2/S_4\) or \(\S^2/D_6\). The rest of this section is devoted to classifying these ramified covers.
	\end{rem}

	We begin by classifying ramified covers of \(\S^2/D_6\). First, we give an upper bound on the number of conical points.

	\begin{lem}\label{res:conePointsBoundForMonodromyD6}
		Let \(S\) be a sphere with conical points and with conical angles at most \(2\pi\) that is a ramified cover of \(\S^2/D_6\). Then \(S\) has at most three conical points.
	\end{lem}
	\begin{proof}
		Let \(\varphi\colon S\surj\S^2/D_6\) be the ramified covering map, and consider the preimage \(\varphi^{-1}([z_1z_2])\), where \([z_1z_2]\) is the shortest path joining the two conical points of angle \(\pi\) in \(\S^2/D_6\). This preimage can be viewed as a finite connected graph embedded in \(S^2\), and is in fact the \emph{dessin d'enfant} for the ramified covering map \(\varphi\) (see \cite[§~2.1]{lando2013graphs}). This graph may be given a 2-colouring by marking whether a vertex maps to \(z_1\) or \(z_2\). The fact that every preimage of \(z_1\) or \(z_2\) has multiplicity at most 2 implies that every vertex of the graph has degree at most 2. The only connected, 2-coloured graphs with maximum degree 2 are circles of even length, or line segments. The preimages of \(z_3\) (the conical point of angle \(\pi/3\)) correspond to the faces of the graph, the multiplicity being half the (graph-theoretic) degree of the face. When the graph is a circle, the only conical points are the two preimages of \(z_3\), in which case \(S\) has two conical points; unless the circle has length 12, in which case \(S\cong\S^2\). The circle cannot have length greater than 12, or else the preimages of \(z_3\) would have conical angle \(>2\pi\). When the graph is a segment, it cannot have length greater than 6, or else the one preimage of \(z_3\) would have conical angle \(>2\pi\). When the segment has length 6, the preimage of \(z_3\) is regular and so \(S\) has two conical points. When the segment has length less than 6, the conical points of \(S\) are the one preimage of \(z_3\) and the two endpoints of the segment.
	\end{proof}

	The above proof gives us a way to construct the possible ramified covers, specifically by finding graphs that could correspond to \(\varphi^{-1}([z_1z_2])\). This technique is used more thoroughly when considering \(\S^2/S_4\). Since we are only interested in spheres with three or more conical points, we are restricted to the five cases where \(\varphi^{-1}([z_1z_2])\) is a segment of length at most 5. We summarise this in the following result.

	\begin{prop}\label{res:linksWithMonodromyInD6}
		Up to isometry, there are five spheres with at least three conical points and with conical angles at most \(2\pi\) that are ramified covers of \(\S^2/D_6\). They are \(\S^2(n\pi/3,\pi,\pi)\), for all \(n\in\{1,\ldots,5\}\).
	\end{prop}

	We now move on to the classification of ramified covers of \(\S^2/S_4\). The same \emph{dessins d'enfants} technique as before is used, but this time it is combined with more involved considerations of the ramification data. To elaborate on this, suppose that \({\varphi\colon S\surj\S^2/S_4}\) is a ramified cover and that \(y_1\), \(y_2\), and \(y_3\) are the conical points of \(\S^2/S_4\) of angle \(\pi\), \(2\pi/3\), and \(\pi/2\) respectively. The conical angles of \(S\) allow us to calculate its area using the spherical Gauss--Bonnet formula (for which, see \cite[Thm.~3.15]{cooper2000three}) and therefore the degree of \(\varphi\). On the other hand, if the preimages \(x_i^1,\ldots,x_i^{k_i}\) of \(y_i\) have multiplicities \(m_i^1,\ldots,m_i^{k_i}\) respectively, then they have conical angles \(m_i^1\alpha_i,\ldots,m_i^{k_i}\alpha_i\) respectively (where \(\alpha_i\) is the conical angle of \(y_i\)), and, by \cref{eq:degMult} the degree of \(\varphi\) is \(m_i^1+\ldots+m_i^{k_i}\). The list of multiplicities of the preimages of \(y_1\), \(y_2\), and \(y_3\) is called a \emph{multiplicity datum}, and calculations of the degree of \(\varphi\) these two different ways we refer to as \emph{area-multiplicity calculations}. These two calculations allow us to exclude certain multiplicity data from appearing, thereby making it feasible to list all relevant ramified covers. This is explained in more detail later, but first, as before, we begin with a bound on the number of conical points.

	\begin{lem}\label{res:conePointsBoundForMonodromyS4}
		Let \(S\) be a sphere with conical points and with conical angles at most \(2\pi\) which is a ramified cover of \(\S^2/S_4\). Then \(S\) has at most five conical points.
	\end{lem}
	\begin{proof}
		Let \(x_1,\ldots,x_n\) be the conical points of \(S\), having conical angles \(\alpha_1,\ldots,\alpha_n\) respectively. We know that \(S\) has some positive area \(A_S\), and we know from \cref{res:integralConicalAngles} that \(\alpha_i\leq3\pi/2\) for all \(i\in\{1,\ldots,n\}\). Furthermore, \(S\) satisfies the following spherical Gauss--Bonnet formula, which follows from \cite[Thm.~3.15]{cooper2000three}:
		\[
			A_S+\sum_{i=1}^{n}(2\pi-\alpha_i)=4\pi.
		\]
		This gives the following inequality:
		\begin{align*}
			n\pi/2 \leq \sum_{i=1}^n(2\pi-\alpha_i) &= 4\pi - A_S,\\
						    &< 4\pi,\\
						  n &< 8.
		\end{align*}

		Thus, \(n\) is at most 7. Suppose first that \(n\) \emph{is} \(7\). Then by Gauss--Bonnet, the possible values for the \(\alpha_i\) are as follows. Firstly, \(\alpha_i=3\pi/2\) for all \(i\in\{1,\ldots,7\}\)---label such a sphere by \(S_1\). Secondly, \(\alpha_i=3\pi/2\) for all \(i\in\{1,\ldots,6\}\) and \(\alpha_7=4\pi/3\)---label this \(S_2\). Lastly, \(\alpha_i=3\pi/2\) for all \(i\in\{1,\ldots,5\}\) and \(\alpha_6=\alpha_7=4\pi/3\)---label this \(S_3\). We will show that none of these spheres can in fact occur as ramified covers of \(\S^2/S_4\).

		By Gauss--Bonnet, we have \(A_{S_1}=\pi/2\), which is thrice the area of \(\S^2/S_4\). Hence, the ramified covering map \(\varphi\colon S\surj \S^2/S_4\) has degree 3. But all the \(x_i\) are preimages of \(y_3\), the conical point of angle \(\pi/2\), and so \(y_3\) has at least seven preimages. This is a contradiction.\\
		Similarly, we have \(A_{S_2}=\pi/3\), so \(S_2\) is a degree 2 cover of \(\S^2/S_4\). But \(y_3\) has at least six preimages: a contradication.\\
		Finally, we have \(A_{S_3}=\pi/6\), which implies that \(S_3\cong\S^2/S_4\). But this is clearly false. Thus, we cannot have \(n=7\).

		Suppose now that \(n\) is \(6\). By Gauss--Bonnet, the possible values for the \(\alpha_i\) are as follows. Firstly, \(\alpha_i=3\pi/2\) for all \(i\in\{1,\ldots,6\}\)---label such a sphere by \(S_4\). Secondly, \(\alpha_i=3\pi/2\) for \(i\in\{1,\ldots,5\}\) and \(\alpha_6=4\pi/3\)---label this \(S_5\). Lastly, \(\alpha_i=3\pi/2\) for \(i\in\{1,\ldots,5\}\) and \(\alpha_6=\pi\)---label this \(S_6\). In a similar way to before, we will show that none of these spheres correspond to genuine ramified covers.

		By Gauss--Bonnet, we have \(A_{S_4}=\pi\), so \(S_4\) is a degree 6 cover of \(\S^2/S_4\). But \(y_3\) has six preimages of multiplicity 3, so the degree is at least 18: a contradiction.\\
		Similarly, we have \(A_{S_5}=5\pi/6\), so \(S_5\) is a degree 5 cover of \(\S^2/S_4\). But \(y_3\) has five preimages of multiplicity 3, so the degree is at least 15: a contradiction.\\
		Finally, we have \(A_{S_6}=\pi/2\), so \(S_6\) is a degree 3 cover of \(\S^2/S_4\). But \(y_3\) has at least five preimages: a contradiction.

		Thus, \(n\) is at most 5. But there is indeed a ramified cover of \(\S^2/S_4\) with five conical points, as shown in \cref{fig:5ConePoints}. It has two conical points of angle \(3\pi/2\) and three of angle \(4\pi/3\).
	\end{proof}

	\begin{figure}[!ht]
		\centering
		\includegraphics[width=0.9\textwidth]{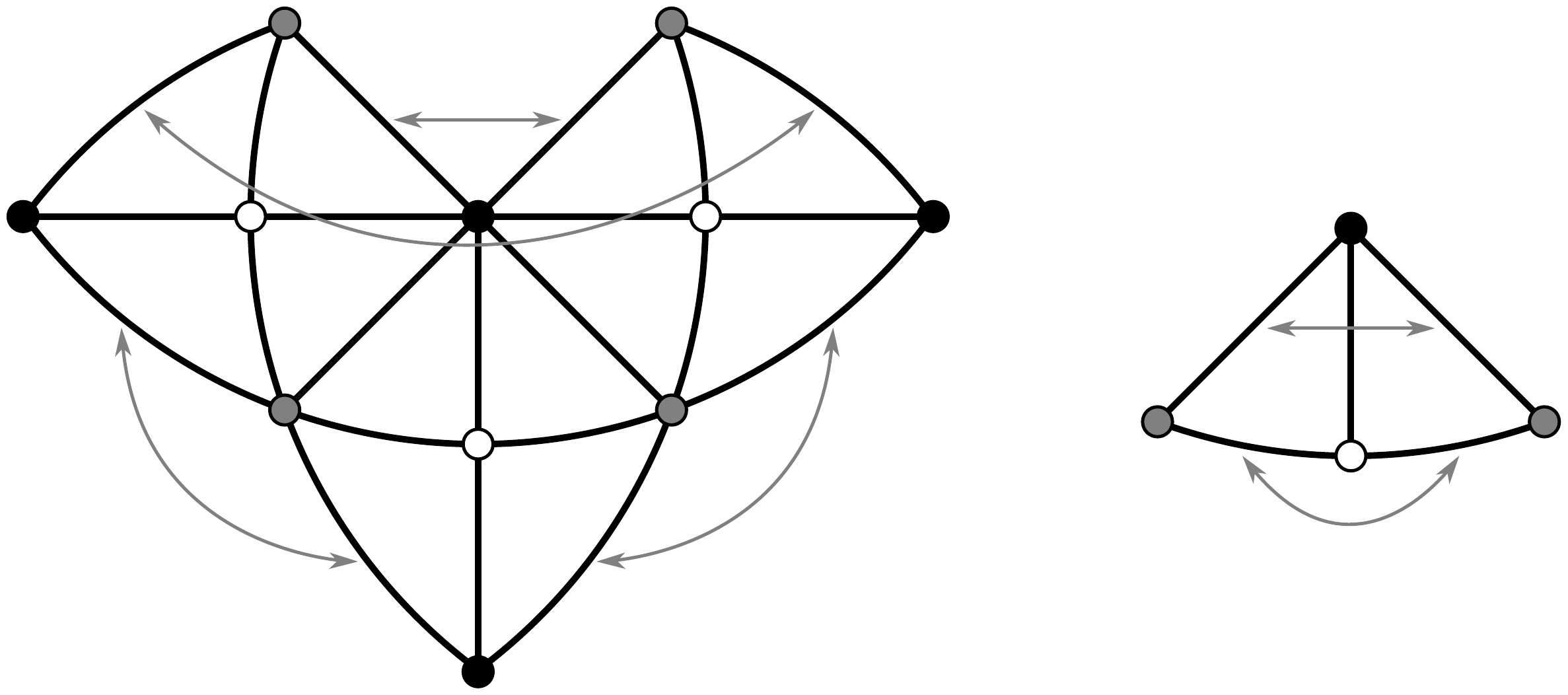}
		\caption{On the left is the unique nonnegatively curved integral link with five conical points. Grey arrows denote edge identifications. It is \emph{not} the double of a spherical polygon. It is constructed out of 12 copies of a spherical triangle, whose double is \(\S^2/S_4\), shown on the right. This demonstrates the ramified covering map. The black, grey, and white vertices of the triangle have interior angles \(\pi/4\), \(\pi/3\), and \(\pi/2\) respectively.}
		\label{fig:5ConePoints}
	\end{figure}

	The above proof gives some useful examples of the area-multiplicity calculations that allow us to exclude certain tuples of conical angles from appearing in ramified covers of \(\S^2/S_4\). We now explain the procedure for finding all the ramified covers of \(\S^2/S_4\) that are spheres with at least 3 conical points and with conical angles at most \(2\pi\).
	
	\begin{proc}\label{proc:findingRamifiedCovers}\hfill
	\begin{enumerate}
		\item For each \(n\in\{3,4,5\}\), we use the Gauss--Bonnet formula and our knowledge of the possible conical angles from \cref{res:integralConicalAngles} to give a finite list of possibilities for the (unordered) tuple \((\alpha_1,\ldots,\alpha_n)\) of conical angles. Explicitly, we solve
		\begin{equation}\label{eq:positivity}
				\sum_{i=1}^{n}(2\pi-\alpha_i)<4\pi,
		\end{equation}
		for \(\alpha_1,\ldots,\alpha_n\in\{\pi/2,2\pi/3,\pi,4\pi/3,3\pi/2\}\).
		\item For any sphere \(S\) with conical points of angle less than \(2\pi\), we can consider the \emph{Dirichlet domain} (for which, see \cite[§~3.6]{cooper2000three}) based at the point with smallest conical angle \(\alpha_{\rm{min}}\). This has the same area as \(S\) and, by \cite[Prop.~3.14]{cooper2000three}, embeds into the spherical football of angle \(\alpha_{\rm{min}}\), which has area \(2\alpha_{\rm{min}}\). Therefore, we exclude tuples of angles that do not satisfy
		\begin{equation}\label{eq:lunes}
			\text{Area} = 4\pi-\sum_{i=1}^{n}(2\pi-\alpha_i)<2\alpha_\rm{min}.
		\end{equation} 
		\item For each remaining tuple \((\alpha_1,\ldots,\alpha_n)\), we look for the possible multiplicity data that could give rise to these angles. Several tuples can be ruled out at this point by area-multiplicity calculations. Explicitly, for a multiplicity datum \(m_j^1,\ldots,m_j^{k_j}\), \(j\in\{1,2,3\}\), we check that
		\begin{equation}\label{eq:degArea}
			\text{Degree} = \sum_{l=1}^{k_j}m_j^l = \frac{1}{\pi/6}\left(4\pi-\sum_{i=1}^{n}(2\pi-\alpha_i)\right) = \text{Area}/\rm{Area}(\S^2/S_4),
		\end{equation}
		for all \(j\in\{1,2,3\}\).
		\item For each remaining multiplicity datum, we look for the possible \emph{dessins d'enfants} \(\varphi^{-1}([y_1y_2])\subseteq S^2\), where \([y_1y_2]\) is the shortest path joining the conical points of angle \(\pi\) and \(2\pi/3\) in \(\S^2/S_4\) and \(\varphi\) is the candidate ramified covering map. Either we construct all possible \emph{dessins} for this datum, in which case we can construct all possible ramified covers with this datum; or we demonstrate that no such \emph{dessin} can exist, in which case this multiplicity datum does not correspond to a genuine ramified cover.
		\item Once all possible ramified covers are listed, we determine whether any of them are isometric. This is only possible when the tuples of conical angles are the same. Moreover, when \(n=3\), isometry type is determined by angles.
	\end{enumerate}
	\end{proc}

	To better illustrate these steps, some examples of the above procedure are given in the \hyperref[exmp:findingRamifiedCovers]{example} in the \hyperref[sec:appendix]{appendix}. We summarise the results of this procedure in the following result. The precise geometries of each possible sphere are too extensive to give here.

	\begin{prop}\label{res:linksWithMonodromyInS4}
		Up to isometry, there are 30 spheres with at least three conical points and with conical angles at most \(2\pi\) that are ramified covers of \(\S^2/S_4\). There are 17 with three conical points, 12 with four conical points, and precisely one with five conical points. With the exception of one pair of nonisometric spheres having conical angles \((\pi,\pi,4\pi/3,4\pi/3)\), these spheres may be distinguished purely by their conical angles.
	\end{prop}

	\begin{rem}
	It is important at this point to note that there are three spheres in common between those listed in \cref{res:linksWithMonodromyInD6,res:linksWithMonodromyInS4}. They are the spheres \(\S^2(2\pi/3,\pi,\pi)\), \(\S^2(\pi,\pi,\pi)\) and \(\S^2(4\pi/3,\pi,\pi)\).
	\end{rem}

	By combining the main results of this section, namely \cref{res:ramifiedCoverForLinks,res:linksWithMonodromyInD6,res:linksWithMonodromyInS4}, we can now fully classify the singular vertices of nonnegatively curved integral polyhedral 3-manifolds. The following result is a summary of the classification. It is an expanded form of \cref{res:classificationOfSingularVertsShort}. A complete geometric description of the links is given in \cref{tab:intVerts} in the \hyperref[sec:appendix]{appendix}.

	\begin{thm}\label{res:classificationOfSingularVerts}
		Let \(M\) be a nonnegatively curved integral polyhedral 3-manifold and \(x\) a singular vertex of \(M\). There are 32 possibilities for the isometry type of \(\link_xM\). Of these, five have (local) monodromy in \(D_6\), 30 in \(S_4\), and three in both. Those with monodromy in \(D_6\) all have three conical points. Of those with monodromy in \(S_4\), 17 have three conical points, 12 have four conical points, and one has five conical points.
	\end{thm}

	\begin{rem}\label{res:classificationOfLinks}
		Given that there are seven possible conical angles in a nonnegatively curved integral polyhedral 3-manifold and that the link of a regular point is \(\S^2\), \cref{res:classificationOfSingularVerts} implies that there are 40 possibilities in total for the local isometry type of a point in a nonnegatively curved integral polyhedral 3-manifold.
	\end{rem}

	We finish with an immediate corollary of this classification, which is useful when studying the global properties of the singular locus.

	\begin{cor}\label{res:linkDiameterBoundedAwayFromPi}
		There is \(\varepsilon_0\in(0,\pi)\) such that, for any singular vertex \(x\) in a nonnegatively curved integral polyhedral 3-manifold \(M\), the diameter of \(\link_xM\) is at most \(\varepsilon_0\); i.e., \(x\) is \(\varepsilon_0\)-narrow.
	\end{cor}
	\begin{proof}
	There are only finitely many possibilities for \(\link_xM\), as just shown, and by \cite[Lem.~3.10]{cooper2000three} any sphere with at least three conical points and with conical angles at most \(2\pi\) has diameter strictly less than \(\pi\).
	\end{proof}


\section{Proof of the Singular Vertex Bound} 
	\label{sec:proof}

	\noindent\Cref{sec:class} dealt with the \emph{local} geometry of singular vertices of nonnegatively curved integral polyhedral 3-manifolds. In this section, we answer a \emph{global} question about singular vertices by proving the following result.
	\singVertBnd*

	As was mentioned in the introduction, the property of nonnegatively curved integral polyhedral 3-manifolds that is used to prove this is the fact that there is \(\varepsilon_0\in(0,\pi)\) such that all singular vertices are \(\varepsilon_0\)-narrow. The proof relies on techniques from Alexandrov geometry---in fact, the majority of what follows applies to any finite dimensional nonnegatively curved Alexandrov space. What we actually prove is the following.

	\begin{prop}\label{res:epsVertBndAlexandrov}
		For all \(n\in\N\) and \(\epsilon\in(0,\pi)\), there is a bound \(B(n,\epsilon)\in\N\) such that, in any (complete, connected) \(n\)-dimensional Alexandrov space of nonnegative curvature, the number of \(\varepsilon\)-narrow points is less than \(B(n,\epsilon)\).
	\end{prop}

	With this result established, we use \cref{res:linkDiameterBoundedAwayFromPi} and the fact, following from \cite[Ex.~2.9~(6)]{MR1185284}, that any nonnegatively curved polyhedral manifold is an Alexandrov space of nonnegative curvature to deduce \cref{res:singVertBnd}.

	\subsection{Connection to Li and Naber} It was mentioned in the introduction that \cref{res:epsVertBndAlexandrov} follows from a recent result of Li and Naber, \cite[Cor.~1.4]{MR4171913}, although this was not known to the author until after the first version of this paper had been written. Let us briefly examine their result, particularly why it implies ours, before summarising the independent proof given in this section. The central concept in \cite{MR4171913} is \emph{quantitative splitting}---a measure of how close a neighbourhood of a point is to being a metric product. We give a summary of the definitions relevant to our setting.

	\begin{defn}[c.f.~{\cite[Defs.~1.1--1.2]{MR4171913}}]\label{def:quantSplitting}
		Let \(M\) be a metric space, \(x\in M\), \(r>0\), \(k\in\N_0\), and \(\delta>0\). We say that \(B_r(x)\) is \emph{\((k,\delta)\)-splitting} if there is a metric space \(Z\) and a point \(p\in\R^k\times Z\) such that \(d_{\mathrm{GH}}(B_r(x),B_r(p))\leq\delta r\), where \(d_\mathrm{GH}\) is the Gromov--Hausdorff distance. We then define the \emph{\((k,\delta)\)-singular set} to be
		\[
			\oldcal S^k_\delta = \oldcal S^k_\delta(M) \defeq \{x\in M \mid \forall r>0, B_r(x)\text{ is not }(k+1,\delta)\text{-splitting}\}.
		\]
	\end{defn}

	These definitions build on the established concept of singular sets. The \emph{singular set} \(\oldcal S(M)\) of an \(m\)-dimensional Alexandrov space \(M\) of curvature bounded below is defined as the set of points whose tangent cones are not isometric to \(\R^m\). It is well known that this singular set has a natural stratification \(\oldcal S(M)=\oldcal S^{m-1}\supseteq\ldots\supseteq\oldcal S^0\), where \(\oldcal S^k\) is the set of points whose tangent cones are not isometric to \(\R^{k+1}\times C\), for any metric cone \(C\). (If \(M\) is a polyhedral manifold, then \(\oldcal S^0\) is the set of singular vertices and \(\oldcal S^1\) is the union of the singular edges.) One can think of \(\oldcal S^k\) as the set of points where \(M\) does not locally look like \(\R^{k+1}\times C\), for any metric cone \(C\). The idea of quantitative splitting strengthens this notion, so that \(\oldcal S^k_\delta\) is the set of points where \(M\) is, in a rigorous sense, \(\delta\) far away from locally looking like \(\R^{k+1}\times C\). As one would expect, we have \(\oldcal S^k=\bigcup_{\delta>0}\oldcal S^k_\delta\).

	The \(k\)-singular set \(\oldcal S^k\) has Hausdorff dimension at most \(k\) and the same is true for \(\oldcal S^k_\delta\). In particular, \(\oldcal S^0_\delta\) is a discrete set of points. Li and Naber's result \cite[Cor.~1.4]{MR4171913} gives a bound on the \(k\)-dimensional Hausdorff measure of \(\oldcal S^k_\delta\cap B_1(p)\) in any Alexandrov space \(M\) of curvature \(\geq-1\) and for any \(p\in M\), depending only on \(\delta\) and the dimension of \(M\). When \(k=0\), this implies a bound on the number of points in \(\oldcal S^0_\delta\cap B_1(p)\). \Cref{res:epsVertBndAlexandrov} concerns \(\varepsilon\)-narrow points, which always belong to \(\oldcal S^0\), in \emph{nonnegatively curved} Alexandrov spaces, so we now restrict to the case of \(k=0\) and curvature \(\geq0\).

	We first remove the dependence on \(B_1(p)\) in the bounds given above by rescaling the metric. For any positive scale \(\lambda\), if \(M\) has curvature \(\geq0\), then so does \(\lambda M\). We also have \(\oldcal S^0_\delta(\lambda M)=\lambda\oldcal S^0_\delta(M)\). Any finite subset of \(\oldcal S^0_\delta\) can thus be rescaled to lie within \(B_1(p)\) for some \(p\in M\). This means that, when \(M\) has curvature \(\geq0\), \cite[Cor.~1.4]{MR4171913} gives a bound on the size of any finite subset of \(\oldcal S^0_\delta\) depending only on \(\delta\) and the dimension of \(M\). This same bound then immediately applies to \(|\oldcal S^0_\delta|\) itself. Therefore, showing that \cite[Cor.~1.4]{MR4171913} implies \cref{res:epsVertBndAlexandrov} reduces to showing that, for any \(\varepsilon\in(0,\pi)\), there is \(\delta=\delta(\varepsilon)>0\) such that any \(\varepsilon\)-narrow point \(x\in M\) belongs to \(\oldcal S^0_\delta\). This is fairly simple: one can use Toponogov's theorem to show that, for any \(r>0\), \(\diam B_r(x)\leq\varphi r\), where \(\varphi\) depends only on \(\varepsilon\); and then recall from \cite[Ex.~7.3.14]{burago2001course} that \(d_\mathrm{GH}(Y,Y')\geq\frac{1}{2}|\diam Y-\diam Y'|\), for bounded metric spaces \(Y\) and \(Y'\).

	Li and Naber's result therefore \emph{does} imply \cref{res:epsVertBndAlexandrov}, but nowhere near the full force of their result is needed in our specialised setting. We now present our independent proof, which is (naturally) considerably shorter than the proof of \cite[Thm.~1.3]{MR4171913}, from which \cite[Cor.~1.4]{MR4171913} follows.

	\subsection{The Proof} The skeletal logic of the proof of \cref{res:epsVertBndAlexandrov} is as follows. Recall that, for \(\varepsilon\in(0,\pi)\), a point \(x\) in an \(n\)-dimensional nonnegatively curved Alexandrov space \(M\) is called \emph{\(\varepsilon\)-narrow} if any angle subtended at \(x\) is at most \(\varepsilon\). This is equivalent to the space of directions at \(x\) having diameter at most \(\varepsilon\). We will show that any sufficiently large subset \(S\subseteq M\) always contains three distinct points forming a triangle with at least one angle greater than \(\varepsilon\). `Sufficiently large' here depends only on \(n\) and \(\varepsilon\). Therefore, if we take \(S\) to be the set of all \(\varepsilon\)-narrow points in \(M\), \(S\) cannot be `sufficiently large'---i.e., it must be smaller than a certain bound depending only on \(n\) and \(\varepsilon\).

	In order to demonstrate that any sufficiently large set \(S\) admits such a triangle, we first show that \(S\) contains a sufficiently long sequence of points \(x_1,\ldots,x_m\) whose consecutive distances decay like a geometric progression. We then take \(m\) large enough to ensure that, for some \(i<j<m\), the angle \(\angle x_ix_mx_j\) is less than \((\pi-\varepsilon)/2\). We then use Toponogov's theorem (for which, see \cite[Thm.~10.3.1]{burago2001course}) to compare \(\triangle x_ix_jx_m\) with a Euclidean triangle. The geometric decay of the sequence ensures that the side ratio \(|x_jx_m|/|x_ix_j|\) is small, which allows us to deduce that \(\angle x_ix_jx_m>\varepsilon\).

	The aforementioned sequence \(x_1,\ldots,x_m\in S\) is constructed by an iterative process, which relies on a certain series of nets for subsets of \(S\). To determine how large \(S\) needs to be to continue this process, we begin by bounding the cardinality of these nets.

	\begin{lem}\label{res:covering}
		Let \(M\) be an \(n\)-dimensional nonnegatively curved Alexandrov space, \({\alpha\in(0,1)}\), and \(X\) a finite subset of \(M\) with diameter at most \(d\). There is a covering of \(X\) by closed balls of radius \(\alpha d/4\) with at most \((1+8/\alpha)^n\) elements.
	\end{lem}
	\begin{proof}
		This is adapted from a proof of Liu in \cite[§~3]{MR1068127}, although the argument is generally attributed to Gromov. Let \(\{p_1,\ldots,p_k\}\) be a maximal set of points in \(X\) satisfying \(d(p_i,p_j)>\alpha d/4\) for all \(i\neq j\). Then we have a covering of \(X\) by, say, \(k\) closed balls:
		\[
			X \subseteq \bigcup_{i=1}^{k}\overline{B}(p_i,\alpha d/4).
		\]
		By the definition of the \(p_i\), we have
		\begin{equation}\label{eq:disjoint}
			\overline{B}(p_i,\alpha d/8)\cap\overline{B}(p_j,\alpha d/8)=\emptyset,\text{ for all } i\neq j.
		\end{equation}
		Also, by applying the triangle inequality, we see that, for any \(j\in\{1,\ldots,k\}\),
		\begin{equation}\label{eq:contained}
			\bigcup_{i=1}^{k}\overline{B}(p_i,\alpha d/8) \subseteq \overline{B}(p_j,(1+\alpha/8)d).
		\end{equation}

		Now let \(\overline{B}(p_j,\alpha d/8)\) have minimal volume among the \(\overline{B}(p_i,\alpha d/8)\). Then, writing \(V(p,r)\defeq\Vol(B(p,r))\), we have
		\begin{align*}
			k &= \sum_{i=1}^{k}\frac{V(p_i,\alpha d/8)}{V(p_i,\alpha d/8)},\\
			  &\leq \frac{1}{V(p_j,\alpha d/8)}\sum_{i=1}^{k}V(p_i,\alpha d/8),\\
			  &\leq \frac{V(p_j,(1+\alpha/8)d)}{V(p_j,\alpha d/8)},\text{ by \cref{eq:disjoint,eq:contained} together,}\\
			  &\leq \frac{(1+\alpha/8)^nd^n}{(\alpha d/8)^n},\text{ by the Bishop--Gromov inequality (i.e., \cite[Thm.~10.6.6]{burago2001course}),}\\
			  &= (1+8/\alpha)^n.\qedhere
		\end{align*}
	\end{proof}

	\noindent We now describe one step of the iterative process.

	\begin{lem}\label{res:shrinkingSequenceStep}
		Let \(M\) be an \(n\)-dimensional nonnegatively curved Alexandrov space, \({\alpha\in(0,1)}\), and \(S\) a finite subset of \(M\) with \(|S|\geq2\) and \(\diam S=d\). There exist \(x\in S\) and \(S'\subseteq S\) such that \(d(x,S')\geq d/2\), \(\diam S'\leq \alpha d/2\), and \(|S'|\geq|S|/\left(2\lfloor(1+8/\alpha)^n\rfloor\right)\).
	\end{lem}
	\begin{proof}
		Choose \(p\) and \(q\in S\) such that \(d(p,q)=d\) and define two subsets of \(S\):
		\[
			P \defeq \setb{s\in S}{d(s,p)\leq d(s,q)}\ \text{and}\ Q \defeq S\take P.
		\]

		If \(|P|\geq|Q|\), let \(x\defeq q\) and \(X\defeq P\); otherwise, let \(x\defeq p\) and \(X\defeq Q\). We note three facts about \(X\). Firstly, the definition of \(X\) and the triangle inequality imply that \(d(x,X)\geq d/2\); secondly, \(\diam X\leq d\); and thirdly, since \(X\) is chosen to be the larger of \(P\) and \(Q\), we have 
		\begin{equation}\label{eq:sizeX}
			|X|\geq(|P|+|Q|)/2=|S|/2.
		\end{equation}
		
		Now let \(k\defeq\lfloor(1+8/\alpha)^n\rfloor\). By \cref{res:covering}, we can cover \(X\) by \(k\) balls \(B_1,\ldots,B_k\) of radius \(\alpha d/4\) in \(M\). Now let \(X\cap B_j\) have maximal cardinality among the \(X\cap B_i\), and let \(S'\defeq X\cap B_j\). Let us demonstrate that \(S'\) has the three required properties. Firstly, \(d(x,S')\geq d(x,X)\geq d/2\); secondly, \(\diam S'\leq\diam B_j\leq\alpha d/2\); and finally, since \(S'\) is chosen to maximise \(|X\cap B_i|\), we have
		\begin{align*}
			|S'| \geq \frac{1}{k}\sum_{i=1}^{k}|X\cap B_i| &\geq \frac{|X|}{k},\\
			&\geq \frac{|S|}{2k},\text{ by \cref{eq:sizeX}},\\
			&= \frac{|S|}{2\lfloor(1+8/\alpha)^n\rfloor}.\qedhere
		\end{align*}
	\end{proof}

	This step is now applied recursively to produce a sequence of points whose consecutive distances decay like a geometric progression.

	\begin{lem}\label{res:shrinkingSequence}
		Let \(M\) be an \(n\)-dimensional nonnegatively curved Alexandrov space, \(m\) an integer greater than 2, and \(\alpha\in(0,1)\). Any subset \(S\subseteq M\) with \(|S|\geq2\left(2\lfloor(1+8/\alpha)^n\rfloor\right)^{m-2}\) contains a sequence of distinct points \(x_1,\ldots,x_m\) satisfying
		\begin{equation}\label{eq:shrinkingSequence}
			d(x_{i+1},x_{i+2})\leq\alpha d(x_i,x_{i+1}),~\text{ for }i\in\{1,\ldots,m-2\}.
		\end{equation}
	\end{lem}
	\begin{proof}
		If \(S\) is infinite, we may replace \(S\) with any finite subset of cardinality at least \(2\left(2\lfloor(1+8/\alpha)^n\rfloor\right)^{m-2}\). So assume that \(S\) is finite. We then make \(m-2\) applications of \cref{res:shrinkingSequenceStep} to \(S\). Specifically, starting with \(S_1\defeq S\) and defining \(S_{i+1}\defeq S_i'\), we get \(x_i\in S_i\) for \(i\in\{1,\ldots,m-2\}\). Then, letting \(k\defeq\lfloor(1+8/\alpha)^n\rfloor\), we have
		\begin{equation*}
			|S_{m-1}|\geq \frac{|S_{m-2}|}{2k}\geq\ldots\geq\frac{|S_1|}{(2k)^{m-2}}\geq\frac{2(2k)^{m-2}}{(2k)^{m-2}}=2.
		\end{equation*}

		Thus, we can arbitrarily choose two distinct points \(x_{m-1}\) and \(x_m\in S_{m-1}\). We thus have our sequence \(x_1,\ldots,x_m\in S\); we just need to check that it satisfies \cref{eq:shrinkingSequence}. Indeed, let \(i\in\{1,\ldots,m-2\}\). Then
		\begin{align*}
			d(x_{i+1},x_{i+2}) &\leq \diam S_{i+1}, \hspace{-85pt} &&\text{ since } x_{i+1},x_{i+2}\in S_{i+1},\\
							   &\leq \frac{\alpha}{2}\diam S_i, \hspace{-85pt} &&\text{ by \cref{res:shrinkingSequenceStep},}\\
							   &\leq \alpha d(x_i,S_{i+1}), \hspace{-85pt} &&\text{ by \cref{res:shrinkingSequenceStep},}\\
							   &\leq \alpha d(x_i,x_{i+1}), \hspace{-85pt} &&\text{ since } x_{i+1}\in S_{i+1}.\qedhere
		\end{align*}
	\end{proof}

	We must now briefly depart from the linear flow of the argument to give the following technical result, which does not follow from any of the earlier results. It states that a sufficiently large collection of points in an Alexandrov space of curvature at least 1 contains two points that are close to each other. This is shown by using a packing result of Grove and Wilhelm (\cite[Prop.~1.3]{MR1343322}) to reduce the problem to the sphere, where it can be solved using elementary spherical geometry. The result is used in the proof of \cref{res:largeAngle} to ensure that, when our sequence \(x_1,\ldots,x_m\) is long enough, we can find indices \(1\leq i<j<m\) such that the angle \(\angle x_ix_mx_j\) is small.

	\begin{lem}\label{res:spaceOfDirPacking}
		Let \(\Sigma\) be an \((n-1)\)-dimensional Alexandrov space of curvature at least 1 and let \(\delta\in(0,\pi)\). Any collection of \(m-1\) points in \(\Sigma\) contains two at a distance at most \(\delta\), provided that \(m\geq m_n(\delta)\), where
		\begin{equation}\label{eq:mDelta}
			m_n(\delta)\defeq\left\lfloor\frac{2}{I_{\sin^2(\delta/2)}\left(\frac{n-1}{2},\frac{1}{2}\right)}\right\rfloor + 2,
		\end{equation}
		and \(I_t(a,b)\) is the regularised incomplete beta function.
	\end{lem}
	\begin{proof}
		We first show that it suffices to prove the result when \(\Sigma=\S^{n-1}\). To do this, we define the \emph{\(\mathnth{q}\) packing radius} of a compact metric space \(X\) (see \cite{MR1343322}). For a positive integer \(q\), this is the quantity
		\[
			\pack_qX \defeq \frac{1}{2}\max_{x_1,\ldots,x_q\in X}\left(\min_{i<j}(d(x_i,x_j))\right).
		\]
		Given this definition, observe that the following two statements are equivalent:
		\begin{enumerate}
			\item \(\pack_{m-1}X\leq\delta/2\),
			\item For all \(v_1,\ldots,v_{m-1}\in X\), there are indices \(i<j\) such that \(d(v_i,v_j)\leq\delta\).
		\end{enumerate}
		The content of the lemma is that \(m\geq m_n(\delta)\) implies the second statement for \(X=\Sigma\). But \cite[Prop.~1.3]{MR1343322} states that \(\pack_{m-1}\Sigma\leq\pack_{m-1}\S^{n-1}\), and so if the second statement holds for \(X=\S^{n-1}\), then it holds for \(X=\Sigma\). It therefore suffices to prove the lemma for \(\Sigma=\S^{n-1}\), which we now do. Two points \(v_i\) and \(v_j\) in \(\S^{n-1}\) being at a distance at most \(\delta\) is the same as the two closed metric balls of radius \(\delta/2\) centred on \(v_i\) and \(v_j\) intersecting. According to \cite[Eq.~1]{MR2813331}, the (\((n-1)\)-dimensional) area of a metric ball in \(\S^{n-1}\) (there called a \emph{hyperspherical cap}) of radius \(\delta/2\in(0,\pi/2)\) is
		\[
			\frac{1}{2}I_{\sin^2(\delta/2)}\left(\frac{n-1}{2},\frac{1}{2}\right)A_{n-1},
		\]
		where \(A_{n-1}\) is the area of \(\S^{n-1}\). By comparing these areas, it follows that any collection of at least \(m_n(\delta)-1\) hyperspherical caps of radius \(\delta/2\) in \(\S^{n-1}\) must contain an intersecting pair. The result follows.
	\end{proof}

	This completes the technical part of the section. We now demonstrate that, when \(n\) and \(\alpha\) are chosen appropriately, we can find three distinct elements of the above sequence that subtend an angle arbitrarily close to \(\pi\). \Cref{res:epsVertBndAlexandrov} then follows immediately.

	\begin{figure}[!ht]
		\labellist
		\small\hair 5pt
		\pinlabel \(M\) at 57 404
		\pinlabel \(y\) [b] at 358 420
		\pinlabel \(x\) [r] at 64 250
		\pinlabel \(T\) at 300 340
		\pinlabel \(z\) [l] at 510 353
		\pinlabel \(\R^2\) at 57 142
		\pinlabel \(\tilde y\) [b] at 361 163
		\pinlabel \(\tilde x\) [r] at 48 15
		\pinlabel \(\tilde T\) at 309 85
		\pinlabel \(\tilde z\) [l] at 527 102
		\endlabellist
		\centering
		\includegraphics[width=0.64\textwidth]{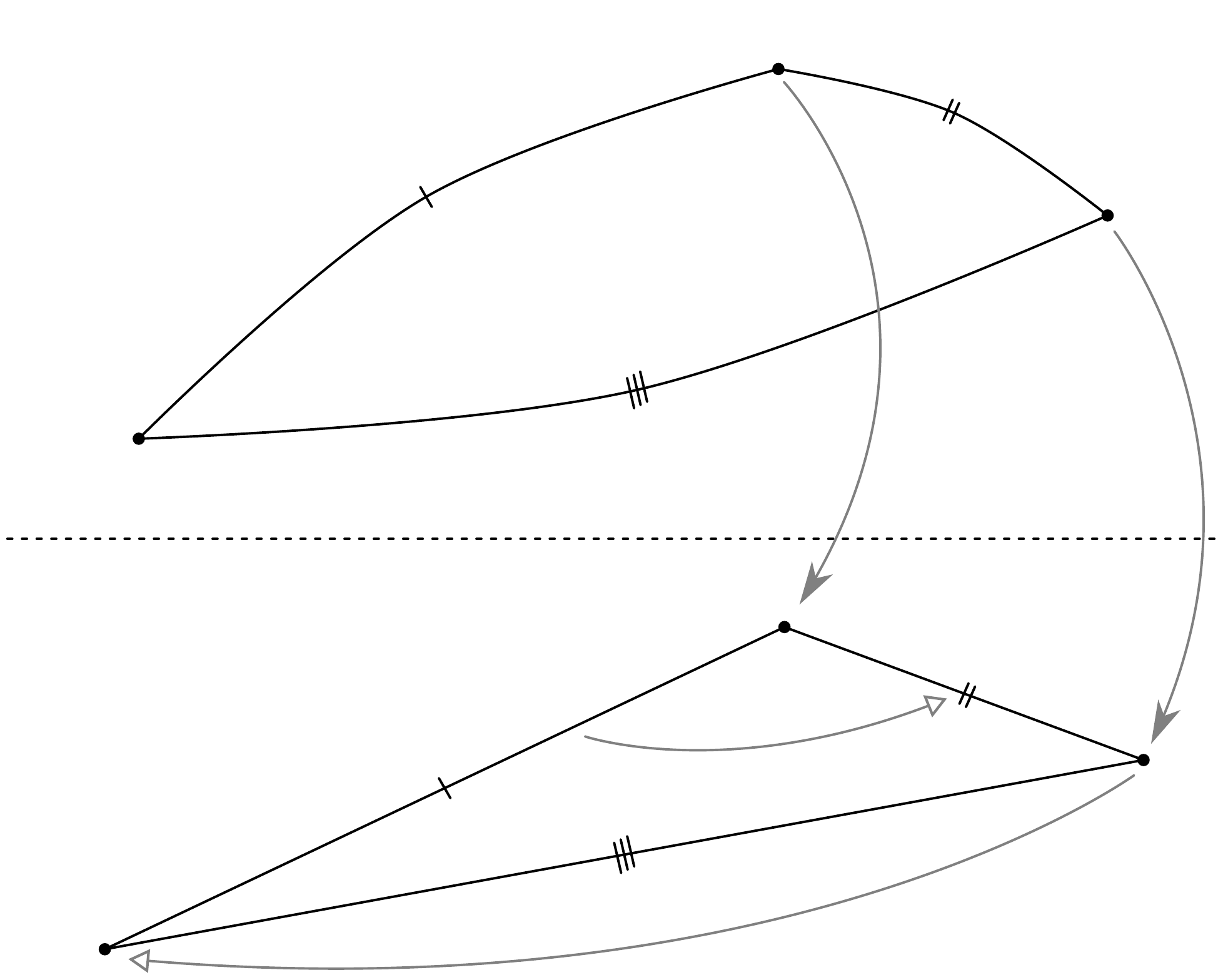}
		\caption{Grey arrows connect one point (or side) to another, signalling that the angle (or length) of the first is at least as big as that of the second. White headed arrows denote \emph{strict} inequalities.}
		\label{fig:triangleComparisons}
	\end{figure}
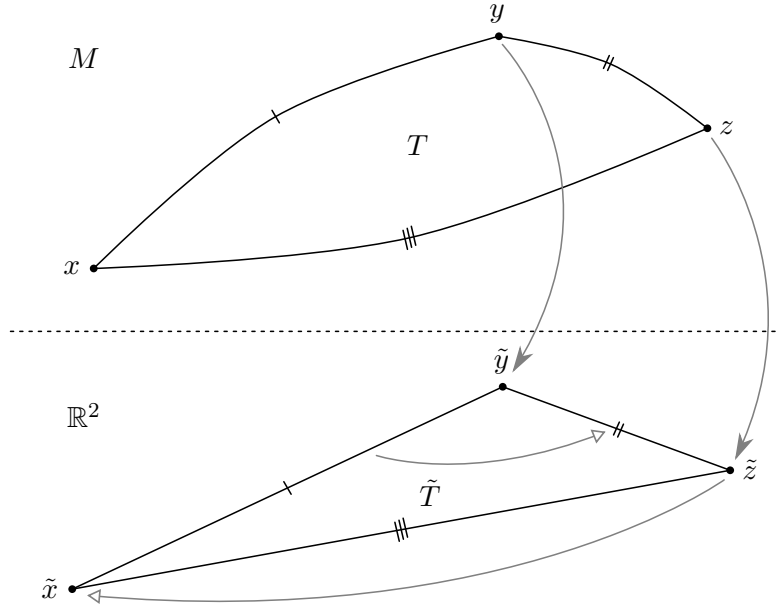

	\begin{lem}\label{res:largeAngle}
		Let \(M\) be an \(n\)-dimensional nonnegatively curved Alexandrov space and let \(\epsilon\in(0,\pi)\). Any subset \(S\) of \(M\) of cardinality at least \(2\cdot\left(2\cdot25^n\right)^{m_n((\pi-\epsilon)/2)-2}\), where \(m_n(-)\) is defined as in \cref{eq:mDelta}, contains three points \(x\), \(y\), and \(z\) such that \(\angle xyz>\epsilon\).
	\end{lem}
	\begin{proof}
		We apply \cref{res:shrinkingSequence} with \(m=m_n((\pi-\epsilon)/2)\) and with \(\alpha=1/3\). We have \(1+8\cdot3=25\), so to do this we need \(|S|\geq2\cdot\left(2\cdot25^n\right)^{m_n((\pi-\epsilon)/2)-2}\). Applying \cref{res:shrinkingSequence} to \(S\) gives a sequence of distinct points \(x_1,\ldots,x_m\) in \(S\) satisfying \cref{eq:shrinkingSequence} with \(\alpha=1/3\). Now, the space of directions \(\Sigma\) at \(x_m\) is, by \cite[Thm.~10.8.6]{burago2001course}, an \((n-1)\)-dimensional Alexandrov space of curvature at least 1. Projecting \(x_1,\ldots,x_{m-1}\) onto \(\Sigma\) via shortest paths, we get a collection of points \(v_1,\ldots,v_{m-1}\) in \(\Sigma\).  By the choice of \(m\), we can apply \cref{res:spaceOfDirPacking} to get indices \(1\leq i<j<m\) such that the distance between \(v_i\) and \(v_j\) in \(\Sigma\) is at most \((\pi-\epsilon)/2\). By definition, this distance is equal to \(\angle x_ix_mx_j\). So take \(x=x_i\), \(y=x_j\), and \(z=x_m\)---then \(\angle xzy\leq(\pi-\varepsilon)/2\). By repeatedly applying the triangle inequality and looking at an appropriate geometric series, we see that \(|yz|=d(x_j,x_n)<d(x_i,x_{i+1})/2<d(x_i,x_j)=|xy|\) (this is ensured by the choice \(\alpha=1/3\)). Finally, we apply Toponogov's theorem (\cite[Thm.~10.3.1]{burago2001course}) to the triangle \(T\defeq\triangle xyz\) (see \cref{fig:triangleComparisons} for a visual explanation of this). Let \(\tilde{T}=\triangle\tilde{x}\tilde{y}\tilde{z}\subseteq\R^2\) be a comparison triangle for \(T\). We have \(\angle xyz\geq\angle\tilde{x}\tilde{y}\tilde{z}\) and \(\angle xzy\geq\angle\tilde{x}\tilde{z}\tilde{y}\) by Toponogov's theorem, and by the Euclidean maxim \emph{``The larger angle is opposite the larger side'',} we have \(\angle\tilde{x}\tilde{z}\tilde{y}>\angle\tilde{y}\tilde{x}\tilde{z}\). Therefore,
		\[
			\angle xyz\geq\angle\tilde{x}\tilde{y}\tilde{z}=\pi-\angle\tilde{x}\tilde{z}\tilde{y}-\angle\tilde{y}\tilde{x}\tilde{z}>\pi-2\angle\tilde{x}\tilde{z}\tilde{y}\geq\pi-2\angle xzy\geq\varepsilon. \qedhere
		\]
	\end{proof}

	\begin{proof}[Proof of \cref{res:epsVertBndAlexandrov}]\label{proof:epsVertBndAlexandrov}
		For \(n\in\N\) and \(\varepsilon\in(0,\pi)\), let \(M\) be an \(n\)-dimensional nonnegatively curved Alexandrov space and let \(S\) be the set of \(\varepsilon\)-narrow points of \(M\). We apply the contrapositive of \cref{res:largeAngle} to \(S\). The angles of any triangle with vertices in \(S\) are at most \(\varepsilon\), so we must have \(|S|<B(n,\varepsilon)\defeq2\cdot\left(2\cdot25^n\right)^{m_n((\pi-\epsilon)/2)-2}\).
	\end{proof}

	We can now specialise back to the case of polyhedral 3-manifolds, by deducing the existence of a singular vertex bound---i.e., by proving \cref{res:singVertBnd}.

	\begin{proof}[Proof of \cref{res:singVertBnd}]
		Let \(M\) be a nonnegatively curved integral polyhedral 3-manifold. Denote by \(\sing{M}^0\) the set of singular vertices of \(M\). \Cref{res:linkDiameterBoundedAwayFromPi} tells us that every point of \(\sing{M}^0\) is \(\varepsilon_0\)-narrow, for some fixed \(\varepsilon_0\). Therefore, since \(M\) is a 3-dimensional nonnegatively curved Alexandrov space (by \cite[Ex.~2.9~(6)]{MR1185284}), we may apply \cref{res:epsVertBndAlexandrov} to deduce that \(|\sing{M}^0|<B_{\rm{ver}}\defeq B(3,\varepsilon_0)\), which is constant.
	\end{proof}

	\begin{rem}
		The vertex bound \(B_{\rm{ver}}\) is by no means effective from a numerical standpoint. Looking at the definition of \(B(n,\varepsilon)\) from the \hyperref[proof:epsVertBndAlexandrov]{proof of \cref*{res:epsVertBndAlexandrov}}, we calculate that \(B_{\rm{ver}}=2\cdot31250^{m_3((\pi-\epsilon_0)/2)-2}\). When \(n=3\), \cref{eq:mDelta} reduces to
		\[
			m_3(\delta) = \left\lfloor \frac{2}{1-\cos(\delta/2)} \right\rfloor + 2,
		\]
		and for integral vertices \(\varepsilon_0\) is at least \(5\pi/6\) (this can be seen by noting that link \#1 in \cref{tab:intVerts} is the double of a triangle having an edge of length \(5\pi/6\)). This bound is therefore at least in the region of \(10^{1051}\). The greatest number of singular vertices in a nonnegatively curved integral polyhedral 3-manifold known to the author is 32. Our aim has therefore been to demonstrate the \emph{existence} of the bound. To produce any useful numerical bound, a different approach must be taken.
	\end{rem}


\section{Further Directions} 
	\label{sec:further}

	\noindent This paper has been concerned with \emph{singular vertices} of polyhedral 3-manifolds—i.e., codimension 3 singularities. As touched upon in the introduction, the term `singular vertex' comes from the following simple fact.
	
	\begin{lem}\label{res:singLocusIsGraphWithLoopsAndLines}
		The connected components of the singular locus of a nonnegatively curved polyhedral 3-manifold are:
		\begin{enumerate}
		\item Locally finite graphs, whose edges have constant conical angle and whose vertices have degree at least 3;
		\item Circles with constant conical angle; or
		\item Lines of infinite length with constant conical angle.
		\end{enumerate}
		When the manifold is compact, case 3 does not occur.
	\end{lem}
	\begin{proof}
	This follows immediately from a generalisation of \cite[Cor.~3.11]{boileau2005geometrization} to the case of conical angles \(<2\pi\).
	\end{proof}

	The vertices of the graph components of the singular locus are precisely the codimension 3 singularities of the polyhedral 3-manifold, i.e., the singular vertices. The conical points of the link of a singular vertex correspond to the edges incident to that vertex. We refer to the edges of the graph components, the circles, and the lines collectively as the \emph{singular edges}.

	A natural extension to what has been done in this paper is to attempt to control the size of the singular locus as a whole, rather than just its vertices. The author believes that the following result holds.

	\begin{conj}[Singular edge bound]\label{res:singEdgeBnd}
		There is a constant \(B\in\N\) such that any (complete, connected) nonnegatively curved integral polyhedral 3-manifold has fewer than \(B\) singular edges.
	\end{conj}

	This result, of course, implies \cref{res:singVertBnd}, but we will probably need to use \cref{res:singVertBnd} to prove it. We know from \cite[Cor.~1.4]{MR4171913} that the sum of the lengths of the singular edges is less than some constant times the diameter of the space, but as yet we have no way to bound the total number. \Cref{res:singEdgeBnd} can also be weakened in another way, this time by neglecting the graph components of the singular locus.

	\begin{conj}[Singular circle/line bound]\label{res:singCircLineBnd}
		There is a constant \(B_\rm{lin}\in\N\) such that the total number of singular circles and lines in any (complete, connected) nonnegatively curved integral polyhedral 3-manifold is less than \(B_\rm{lin}\).
	\end{conj}

	Since, by \cref{res:classificationOfSingularVerts}, the maximum degree of a singular vertex is 5, \cref{res:singVertBnd} implies a bound on the number of edges in the graph components of the singular locus. \Cref{res:singCircLineBnd} would therefore imply \cref{res:singEdgeBnd}.

	Another way to extend \cref{res:singVertBnd} is to generalise it to higher dimensions. The proof did not actually require the full classification of singular vertices, but only that there are finitely many isometry types. If we can show that, for any fixed dimension, there are only finitely many types of singular vertices (i.e., singularities of maximal codimension), then we should be able to apply \cref{res:epsVertBndAlexandrov} to nonnegatively curved integral polyhedral manifolds of any dimension.


\appendix
\section*{Appendix: Integral Vertices} 
	\label{sec:appendix}
	\setcounter{section}{1}
	\setcounter{thm}{0}

	\noindent The aim of this appendix is to give examples of the steps in \cref{proc:findingRamifiedCovers} and then to list the links of every possible singular vertex in nonnegatively curved integral polyhedral 3-manifolds in \cref{tab:intVerts}.

	\begin{exmp*}\label{exmp:findingRamifiedCovers}
	Here we give step-by-step examples of \cref{proc:findingRamifiedCovers} for finding ramified covers. Theoretically speaking, nothing is added here, but the process is more clearly illustrated.
	\begin{enumerate}
		\item Let us consider the case of \(n=4\); i.e., of spheres with four conical points. Steps 1 and 2 are straightforward. There are 32 (unordered) quadruples satisfying \cref{eq:positivity}.
		\item Of the 32 quadruples from step 1, only one of them does not satisfy \cref{eq:lunes}: \(\left(\frac{\pi}{2},\frac{3\pi}{2},\frac{3\pi}{2},\frac{3\pi}{2}\right)\).
		\item One of the remaining quadruples is \(\left(\frac{4\pi}{3},\frac{3\pi}{2},\frac{3\pi}{2},\frac{3\pi}{2}\right)\). Using Gauss--Bonnet, we calculate that a sphere with these conical angles has area \(11\pi/6\) and so is a degree 11 ramified cover of \(\S^2/S_4\). The three conical points of angle \(3\pi/2\) must be multiplicity 3 preimages of the conical point of angle \(\pi/2\). Since the multiplicities must add to 11, there must also be one multiplicity 2 preimage. But this implies the existence a conical point of angle \(\pi\); a contradiction. Thus, this tuple cannot satisfy \cref{eq:degArea}.
		\item One of the 15 remaining quadruples at this stage is \(\left(\pi,\frac{3\pi}{2},\frac{3\pi}{2},\frac{3\pi}{2}\right)\). We calculate that the degree of the cover must be 9, and the only possible multiplicity datum is: \(m_1^1=1,m_1^2=\ldots=m_1^5=2;\) \(m_2^1=m_2^2=m_2^3=3;\) \(m_3^1=m_3^2=m_3^3=3\). The \emph{dessin d'enfant} \(\varphi^{-1}([y_1y_2])\) would be a bipartite graph, with three faces of degree 6 (corresponding to preimages of \(y_3\)), with four vertices marked \(y_1\) of degree 2 and one of degree 1, and with three vertices marked \(y_2\) of degree 3. It can be shown that no such graph embedded in \(S^2\) exists (see \cref{fig:noSigGraph}).
		\item At this stage, we have 12 different ramified covers, only two of which have the same conical angles: \(\left(\pi,\pi,\frac{4\pi}{3},\frac{4\pi}{3}\right)\). They are distinguished by whether or not the points of angle \(\pi\) are preimages of \(y_1\) or \(y_3\). By expanding the \emph{dessins} \(\varphi^{-1}([y_1y_2])\) to full triangulations, we can see that they are doubles of two noncongruent spherical quadrilaterals, and therefore cannot be isometric (see \cref{fig:twoSigGraphs}).
	\end{enumerate}
	\end{exmp*}

	\begin{figure}[!ht]
		\centering
		\includegraphics[width=0.5\textwidth]{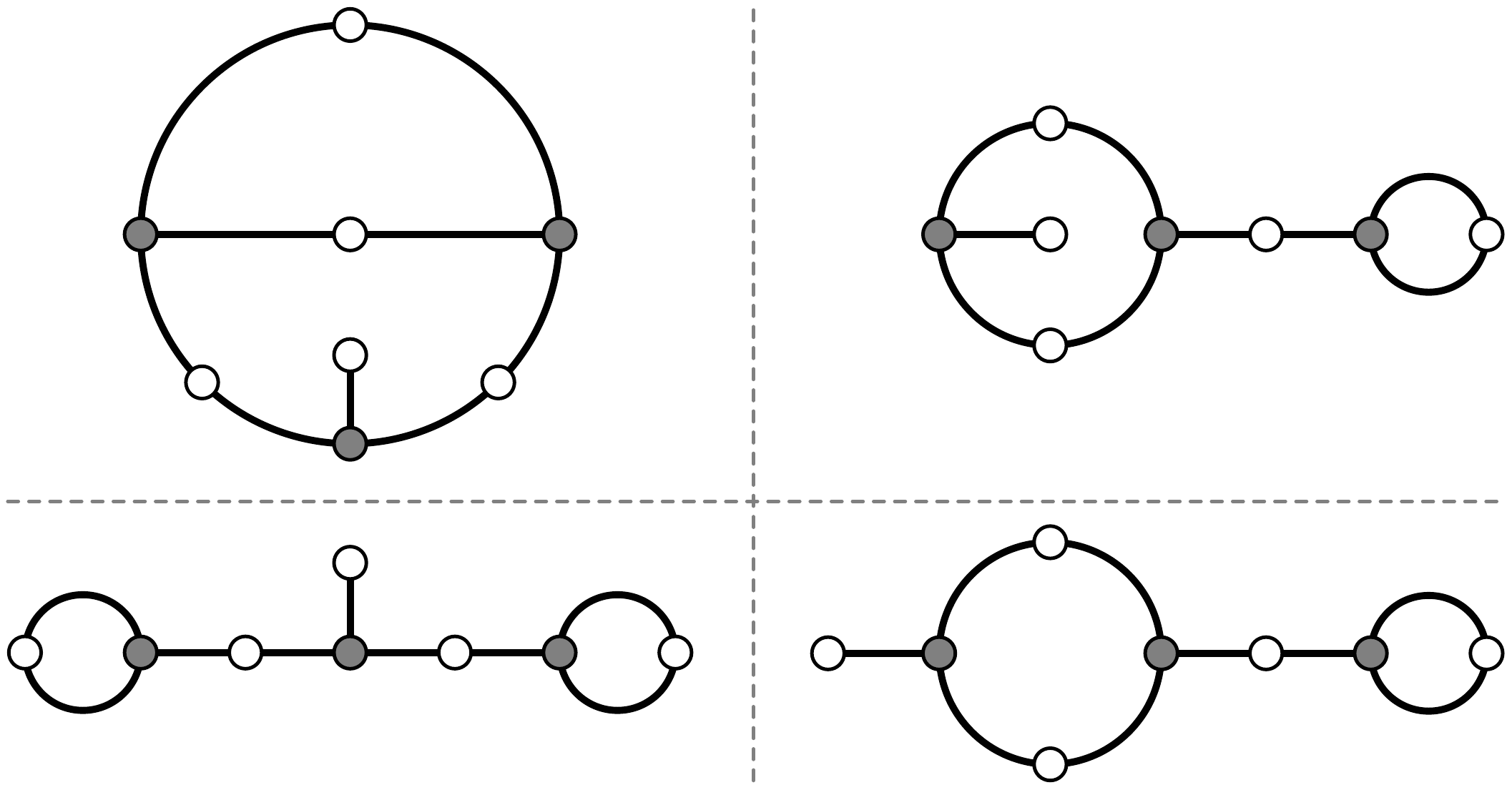}
		\caption{White and grey vertices are marked \(y_1\) and \(y_2\) respectively. This figure demonstrates the search for the \emph{dessin d'enfant} corresponding to the tuple \(\left(\pi,\frac{3\pi}{2},\frac{3\pi}{2},\frac{3\pi}{2}\right)\). If the proposed graph existed, deleting the sole vertex of degree 1 and then temporarily ignoring those of degree 2 leaves a trivalent graph with two grey vertices. There are only two such graphs---the theta graph and the barbell graph. There are four ways to add back in the degree 1 and 2 vertices that give distinct embedded graphs in \(S^2\); these are shown above. None of them has all three faces having degree 6.}
		\label{fig:noSigGraph}
	\end{figure}

	\begin{figure}[!ht]
		\floatbox[{\capbeside\thisfloatsetup{capbesideposition={right,center},capbesidewidth=0.53\textwidth}}]{figure}[\FBwidth]
		{\caption{White, grey, and black vertices are marked \(y_1\), \(y_2\), and \(y_3\) respectively. This figure demostrates the construction of the two nonisometric ramified covers of \(\S^2/S_4\) corresponding to the tuple \(\left(\pi,\pi,\frac{4\pi}{3},\frac{4\pi}{3}\right)\). The two constructions are shown in parallel on the left and the right (corresponding to entries \#23 and \#24 in \cref{tab:intVerts} respectively). The first row gives the two possible multiplicity data. The second row shows the corresponding \emph{dessins d'enfants}. The third row shows the full triangulations of \(S^2\) given by pulling back the triangulation of \(\S^2/S_4\) along the ramified covering maps. At this point, we notice that both are doubles of spherical quadrilaterals, whose boundaries are marked in grey. These spherical quadrilaterals are shown more clearly in the final row.}\label{fig:twoSigGraphs}}
		{\includegraphics[width=0.47\textwidth]{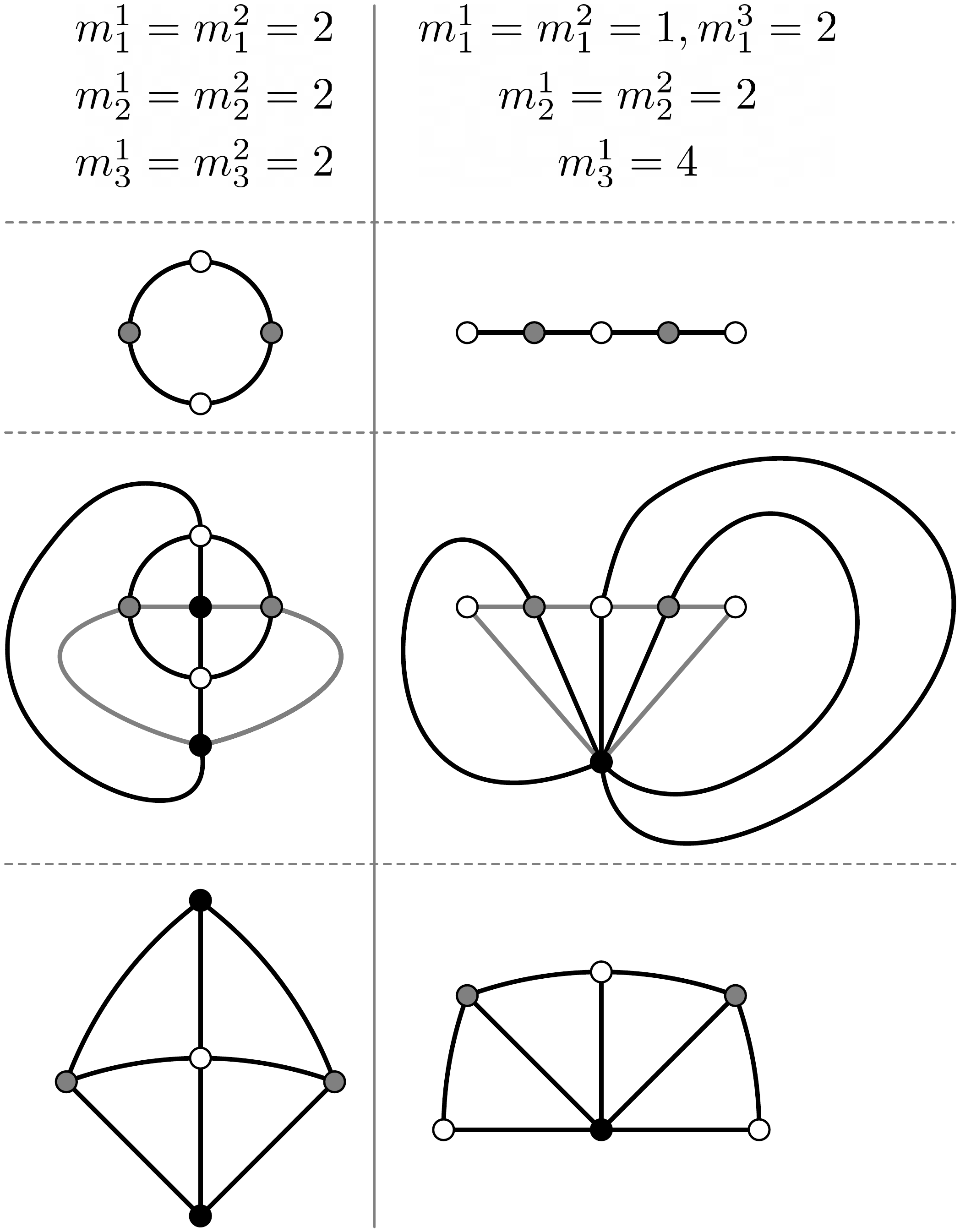}}
	\end{figure}

	We now give \cref{tab:intVerts}, which lists the links of every possible singular vertex in a nonnegatively curved integral polyhedral 3-manifold. Each row in the table corresponds to the link of a unique vertex. The tuple \(\bm\vartheta\) is defined so that \(2\pi\bm\vartheta\) is the list of conical angles of the link in increasing order. The third column says whether or not the link is the double of a spherical polygon. The fourth column gives the local monodromy group. The purpose of the final column is to give a complete geometric description of the link where necessary. If \(\bm\vartheta=(\alpha,\beta,\gamma)\), then this final column is left blank, as the link is completely determined as the double of the unique spherical triangle with angles \(\pi\alpha\), \(\pi\beta\), and \(\pi\gamma\). Otherwise, if the link is the double of a spherical quadrilateral, then that quadrilateral is shown in the final column, built out of copies of the spherical triangle with angles \(\pi/4\), \(\pi/3\), and \(\pi/2\) (denoted by black, grey, and white vertices respectively). If the link is not a double, then a full triangulation is given, with grey arrows denoting edge identifications. The table is ordered increasingly, first by number of conical points, then by size of first conical angle (then second, third, etc.). The only two links with the same conical angles are links \#23 and \#24, which are ordered by size of monodromy group.

	\begin{longtable}[c]{@{}r@{:~}lccc@{}}
	\caption{Links of nonnegatively curved integral vertices\label{tab:intVerts}}\\
	\toprule
	\textbf{\#} & \(\bm\vartheta\)                                                             & \textbf{Double?} & \textbf{Monodromy} & \textbf{Description} \\ \bottomrule \endfirsthead
	\caption[]{(Continued)}\\
	\toprule
	\textbf{\#} & \(\bm\vartheta\)                                                             & \textbf{Double?} & \textbf{Monodromy} & \textbf{Description} \\ \bottomrule \endhead
	1           & \(\left(\frac{1}{6},\frac{1}{2},\frac{1}{2}\right)\)                         & Yes              & \(D_6\)            &                      \\ \midrule
	2           & \(\left(\frac{1}{4},\frac{1}{4},\frac{2}{3}\right)\)                         & Yes              & \(S_4\)            &                      \\ \midrule
	3           & \(\left(\frac{1}{4},\frac{1}{3},\frac{1}{2}\right)\)                         & Yes              & \(S_4\)            &                      \\ \midrule
	4           & \(\left(\frac{1}{4},\frac{1}{3},\frac{3}{4}\right)\)                         & Yes              & \(S_4\)            &                      \\ \midrule
	5           & \(\left(\frac{1}{4},\frac{1}{2},\frac{1}{2}\right)\)                         & Yes              & \(D_4\)            &                      \\ \midrule
	6           & \(\left(\frac{1}{4},\frac{1}{2},\frac{2}{3}\right)\)                         & Yes              & \(S_4\)            &                      \\ \midrule
	7           & \(\left(\frac{1}{3},\frac{1}{3},\frac{1}{2}\right)\)                         & Yes              & \(A_4\)            &                      \\ \midrule
	8           & \(\left(\frac{1}{3},\frac{1}{3},\frac{2}{3}\right)\)                         & Yes              & \(A_4\)            &                      \\ \midrule
	9           & \(\left(\frac{1}{3},\frac{1}{2},\frac{1}{2}\right)\)                         & Yes              & \(D_3\)            &                      \\ \midrule
	10          & \(\left(\frac{1}{3},\frac{1}{2},\frac{2}{3}\right)\)                         & Yes              & \(A_4\)            &                      \\ \midrule
	11          & \(\left(\frac{1}{3},\frac{1}{2},\frac{3}{4}\right)\)                         & Yes              & \(S_4\)            &                      \\ \midrule
	12          & \(\left(\frac{1}{2},\frac{1}{2},\frac{1}{2}\right)\)                         & Yes              & \(D_2\)            &                      \\ \midrule
	13          & \(\left(\frac{1}{2},\frac{1}{2},\frac{2}{3}\right)\)                         & Yes              & \(D_3\)            &                      \\ \midrule
	14          & \(\left(\frac{1}{2},\frac{1}{2},\frac{3}{4}\right)\)                         & Yes              & \(D_4\)            &                      \\ \midrule
	15          & \(\left(\frac{1}{2},\frac{1}{2},\frac{5}{6}\right)\)                         & Yes              & \(D_6\)            &                      \\ \midrule
	16          & \(\left(\frac{1}{2},\frac{2}{3},\frac{2}{3}\right)\)                         & Yes              & \(A_4\)            &                      \\ \midrule
	17          & \(\left(\frac{1}{2},\frac{2}{3},\frac{3}{4}\right)\)                         & Yes              & \(S_4\)            &                      \\ \midrule
	18          & \(\left(\frac{2}{3},\frac{2}{3},\frac{2}{3}\right)\)                         & Yes              & \(A_4\)            &                      \\ \midrule
	19          & \(\left(\frac{2}{3},\frac{3}{4},\frac{3}{4}\right)\)                         & Yes              & \(S_4\)            &                      \\ \midrule
	20          & \(\left(\frac{1}{3},\frac{1}{2},\frac{2}{3},\frac{3}{4}\right)\)             & Yes              & \(S_4\)            & \includegraphics[scale=0.2,valign=t]{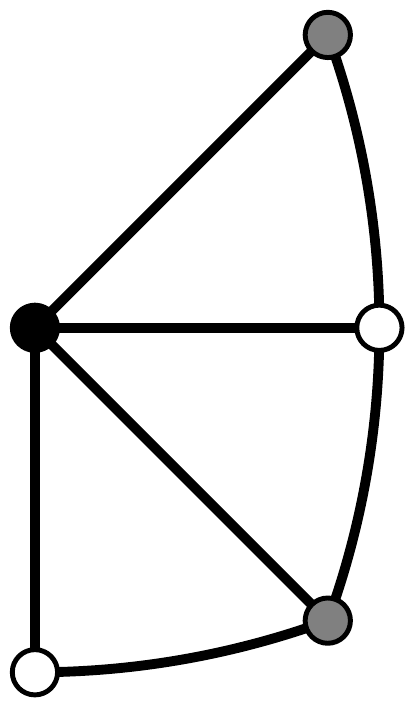}                  \\ \midrule
	21          & \(\left(\frac{1}{2},\frac{1}{2},\frac{1}{2},\frac{2}{3}\right)\)             & Yes              & \(S_4\)            & \includegraphics[scale=0.2,valign=t]{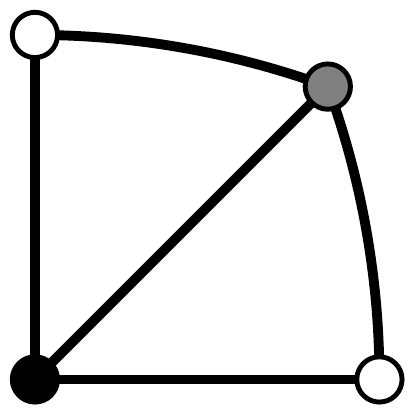}                  \\ \midrule
	22          & \(\left(\frac{1}{2},\frac{1}{2},\frac{1}{2},\frac{3}{4}\right)\)             & No               & \(S_4\)            & \includegraphics[scale=0.2,valign=t]{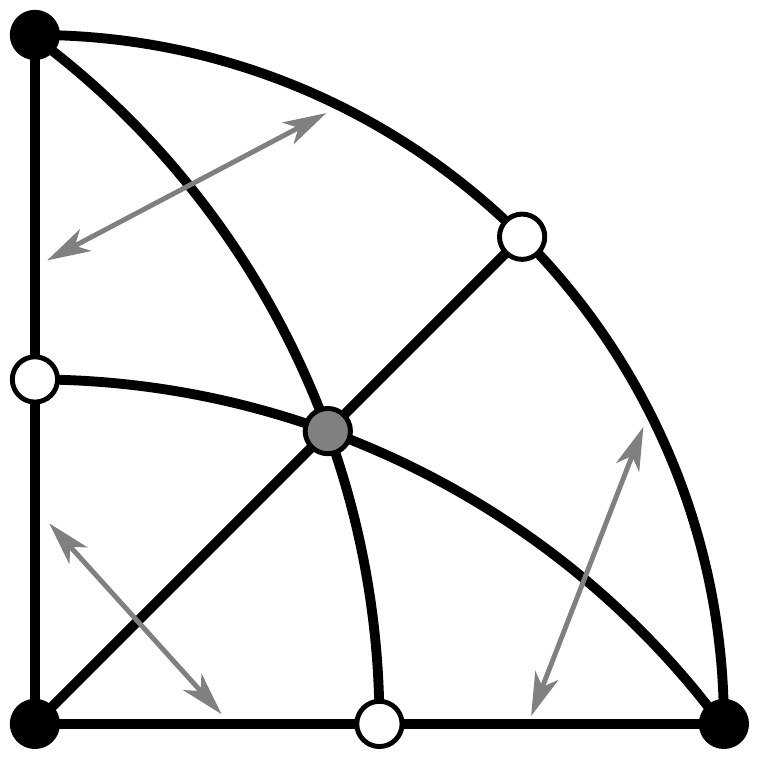}                  \\ \midrule
	23          & \(\left(\frac{1}{2},\frac{1}{2},\frac{2}{3},\frac{2}{3}\right)\)             & Yes              & \(A_4\)            & \includegraphics[scale=0.2,valign=t]{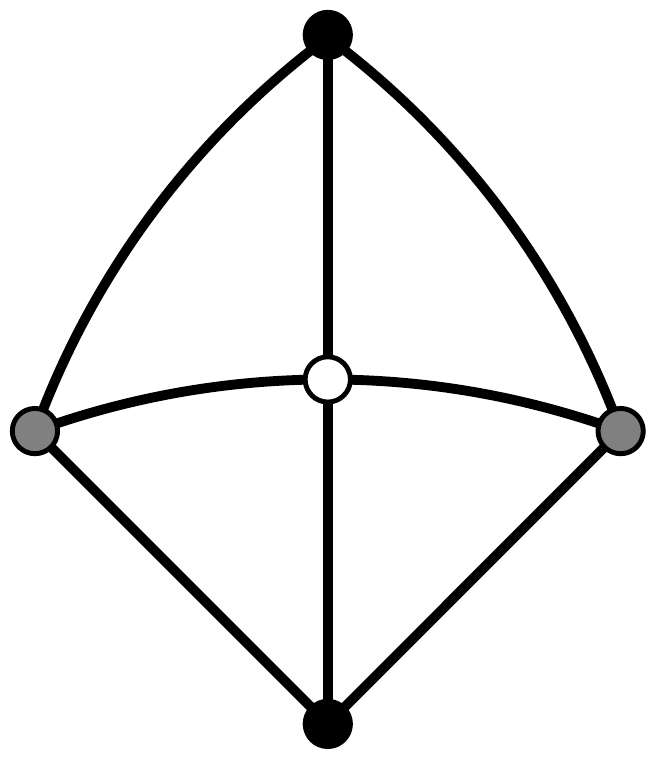}                  \\ \midrule
	24          & \(\left(\frac{1}{2},\frac{1}{2},\frac{2}{3},\frac{2}{3}\right)\)             & Yes              & \(S_4\)            & \includegraphics[scale=0.2,valign=t]{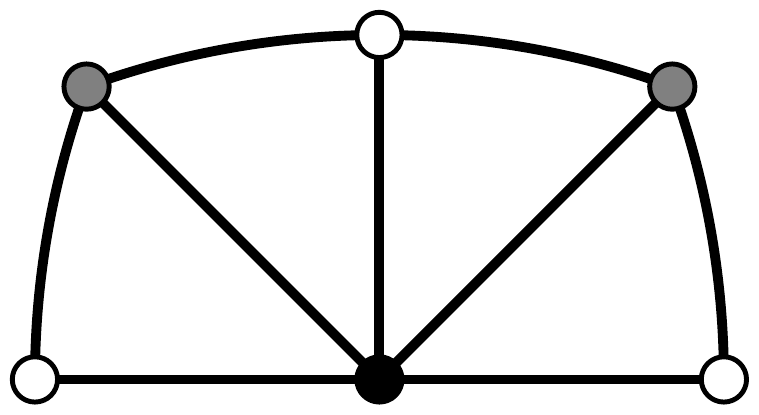}                  \\ \midrule
	25          & \(\left(\frac{1}{2},\frac{1}{2},\frac{2}{3},\frac{3}{4}\right)\)             & Yes              & \(S_4\)            & \includegraphics[scale=0.2,valign=t]{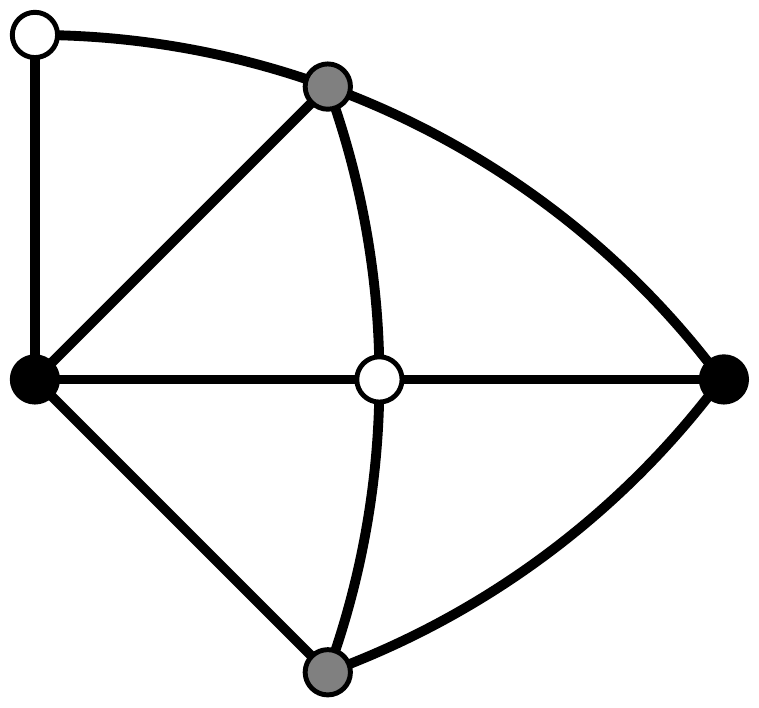}                  \\ \midrule
	26          & \(\left(\frac{1}{2},\frac{1}{2},\frac{3}{4},\frac{3}{4}\right)\)             & Yes              & \(S_4\)            & \includegraphics[scale=0.2,valign=t]{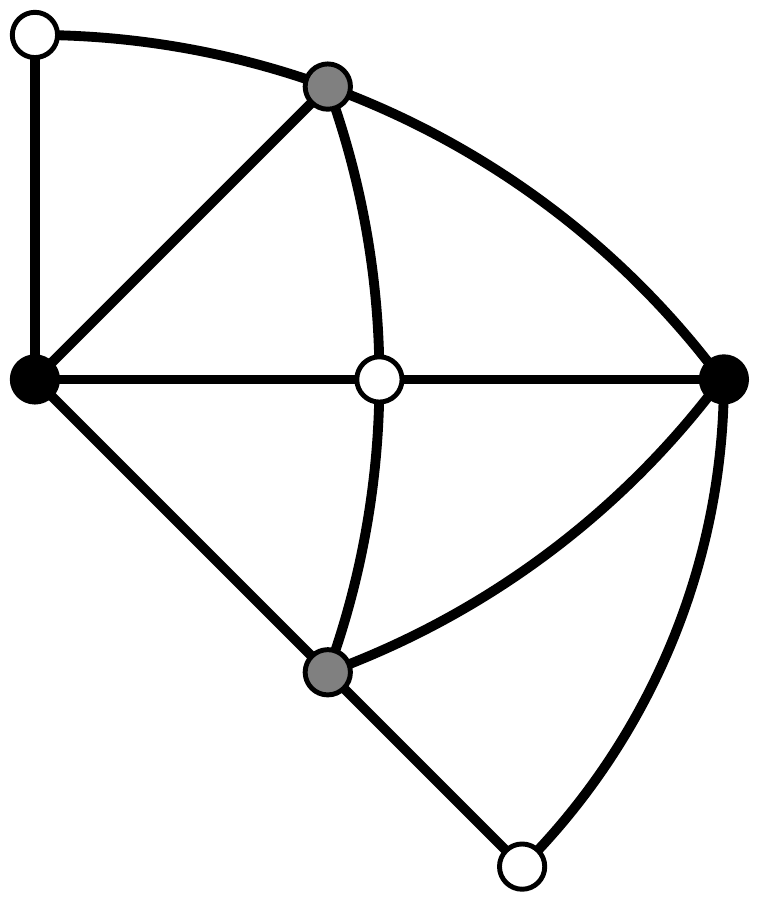}                  \\ \midrule
	27          & \(\left(\frac{1}{2},\frac{2}{3},\frac{2}{3},\frac{3}{4}\right)\)             & No               & \(S_4\)            & \includegraphics[scale=0.2,valign=t]{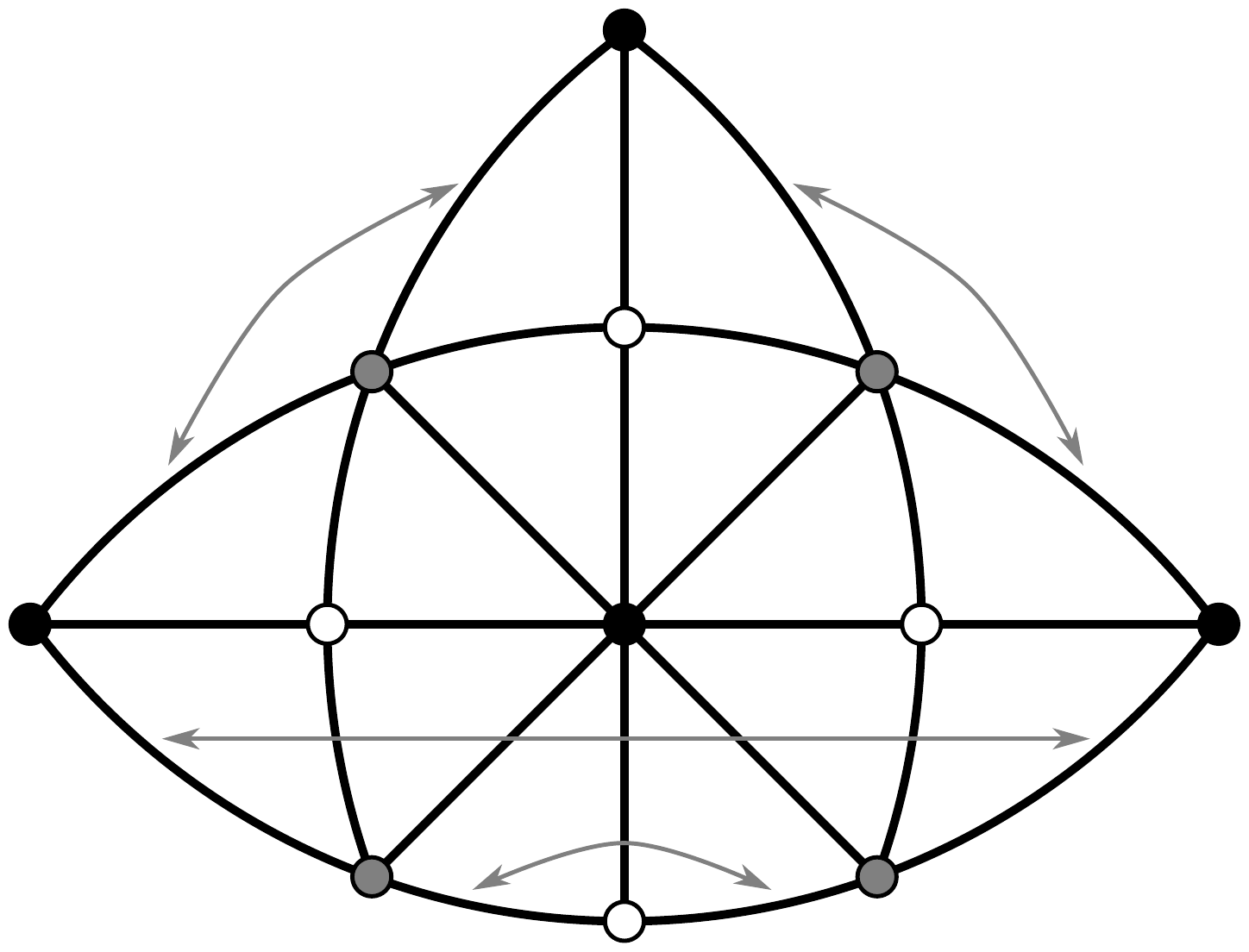}                  \\ \midrule
	28          & \(\left(\frac{1}{2},\frac{2}{3},\frac{3}{4},\frac{3}{4}\right)\)             & Yes              & \(S_4\)            & \includegraphics[scale=0.2,valign=t]{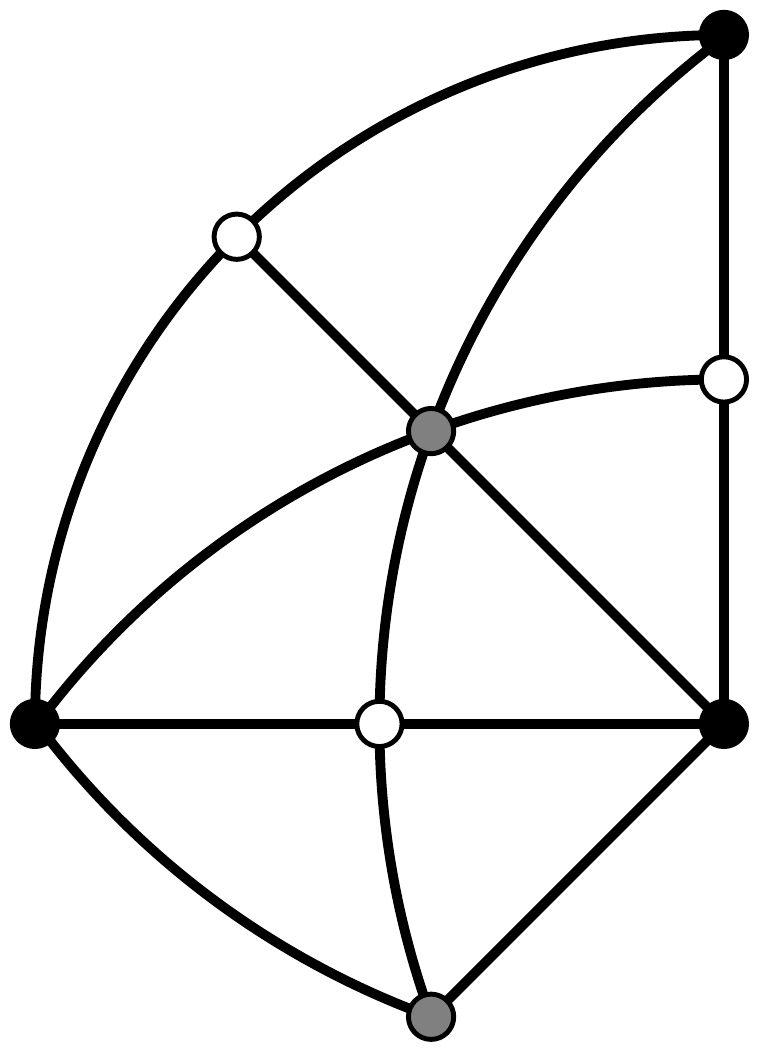}                  \\ \midrule
	29          & \(\left(\frac{2}{3},\frac{2}{3},\frac{2}{3},\frac{2}{3}\right)\)             & Yes              & \(A_4\)            & \includegraphics[scale=0.2,valign=t]{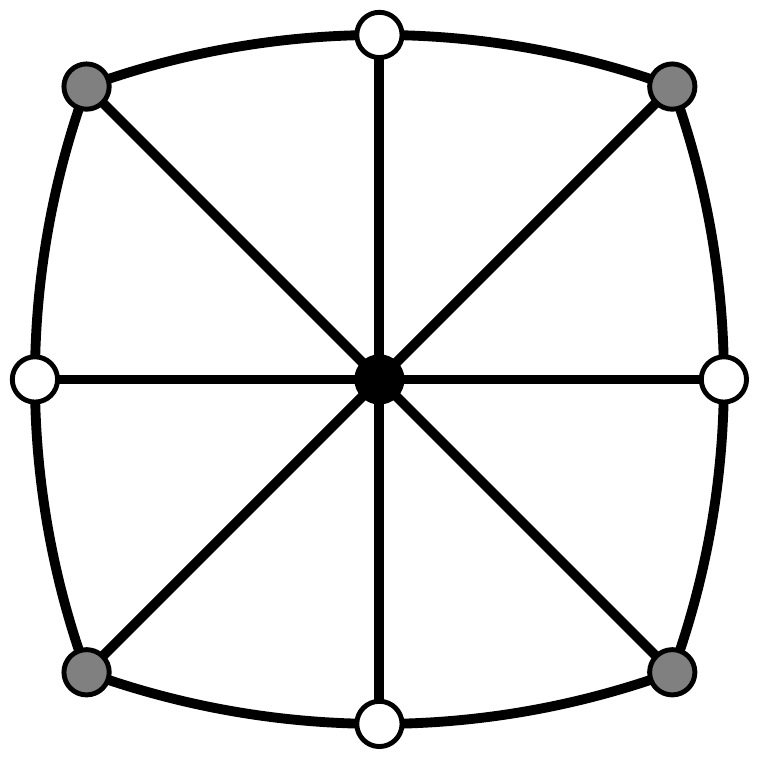}                  \\ \midrule
	30          & \(\left(\frac{2}{3},\frac{2}{3},\frac{3}{4},\frac{3}{4}\right)\)             & Yes              & \(S_4\)            & \includegraphics[scale=0.2,valign=t]{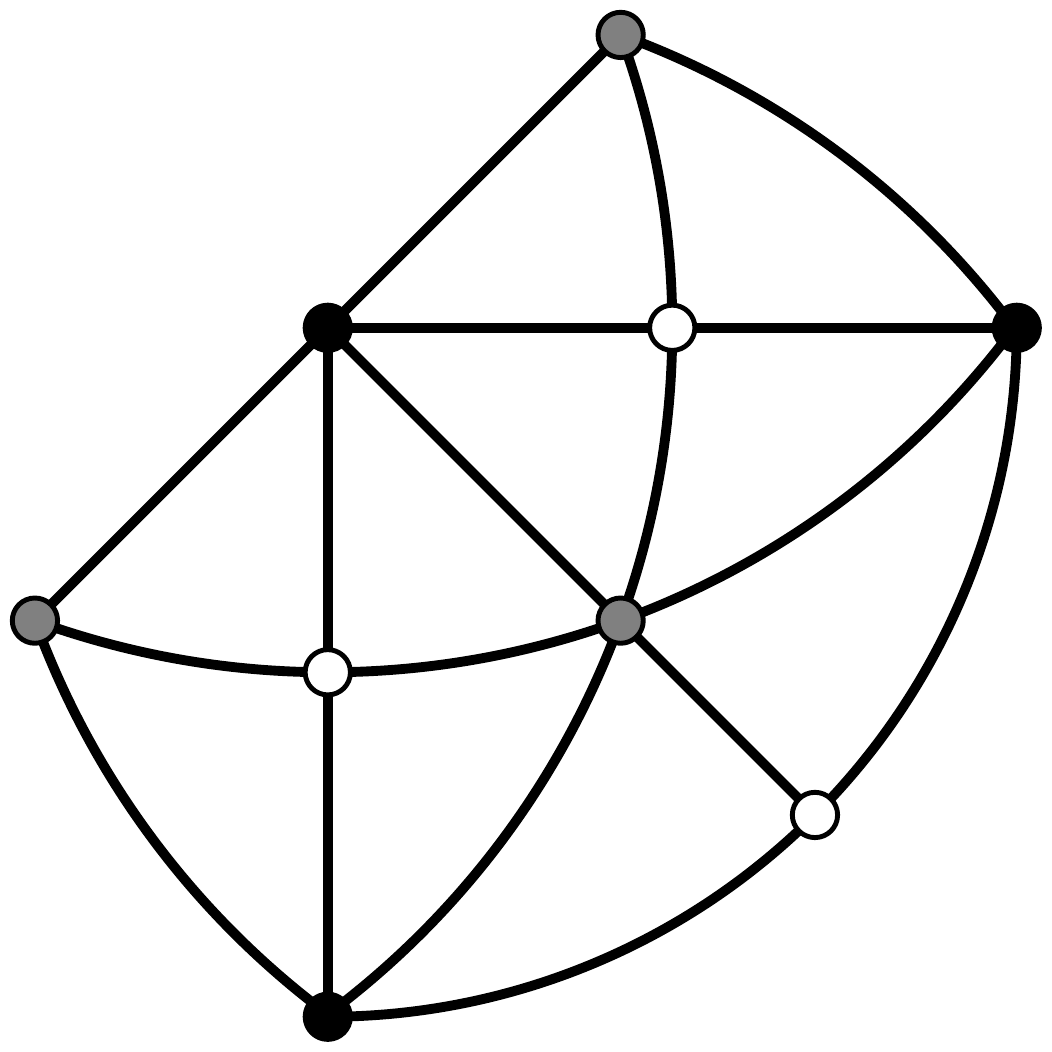}                  \\ \midrule
	31          & \(\left(\frac{3}{4},\frac{3}{4},\frac{3}{4},\frac{3}{4}\right)\)             & No               & \(S_4\)            & \includegraphics[scale=0.2,valign=t]{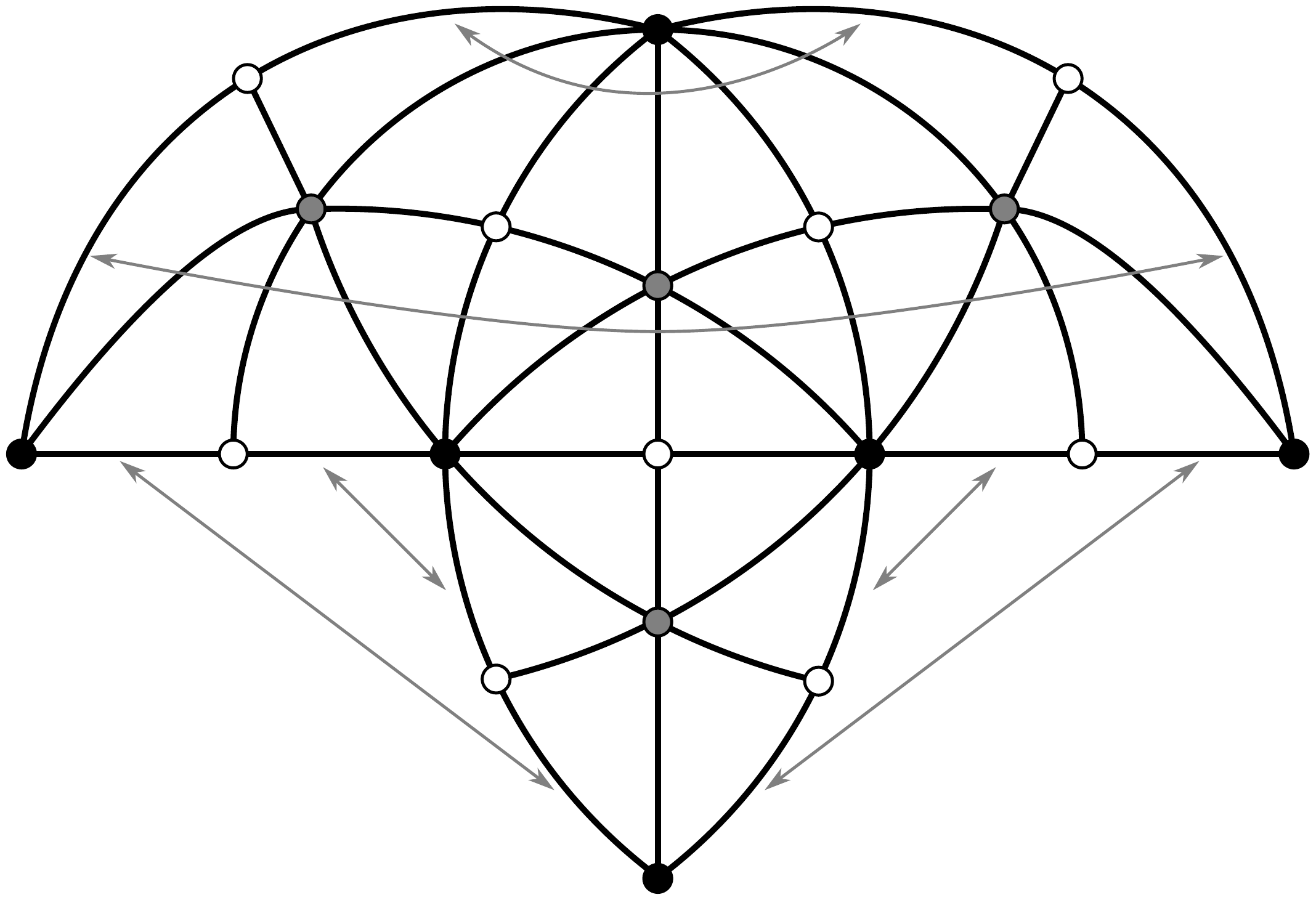}                  \\ \midrule
	32          & \(\left(\frac{2}{3},\frac{2}{3},\frac{2}{3},\frac{3}{4},\frac{3}{4}\right)\) & No               & \(S_4\)            & \includegraphics[scale=0.2,valign=t]{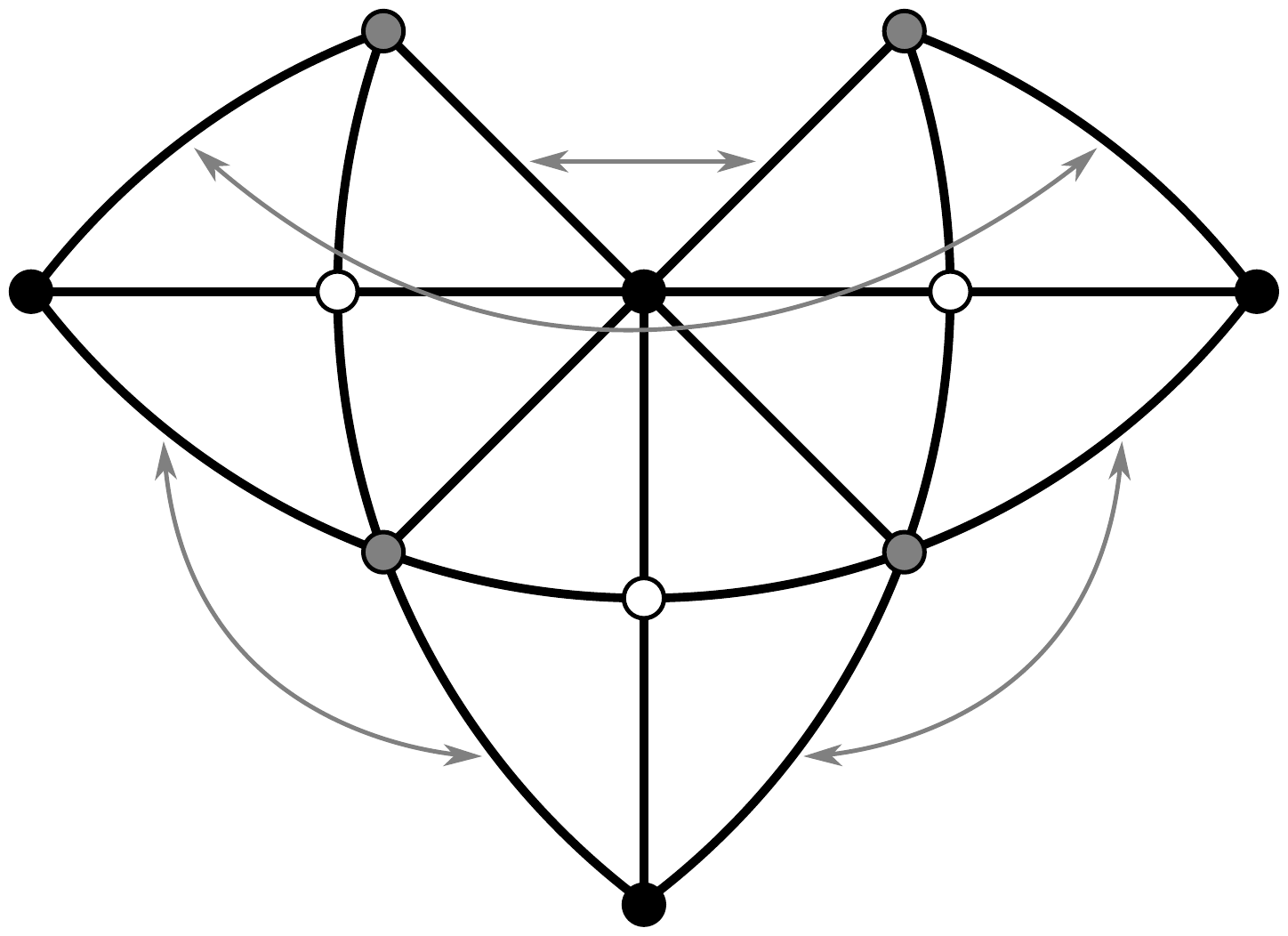}               \\ \bottomrule
	\end{longtable}


\bibliographystyle{Tamsalpha}
\bibliography{main}

\end{document}